\documentclass[11pt]{amsart}

\usepackage{amsthm,amsmath,amssymb,geometry}
\theoremstyle{definition} \newtheorem{cor}{Corollary}[section]
\theoremstyle{definition} \newtheorem{lem}[cor]{Lemma} 
\theoremstyle{definition} \newtheorem{prop}[cor]{Proposition} 
\theoremstyle{definition} \newtheorem{thm}[cor]{Theorem}
\theoremstyle{definition} 
\theoremstyle{definition} \newtheorem{defn}[cor]{Definition}
\theoremstyle{definition} 
\theoremstyle{definition} \newtheorem*{thma}{Theorem}
\theoremstyle{definition}  
\theoremstyle{definition} 
\theoremstyle{definition}  
\theoremstyle{definition}  
\theoremstyle{definition}  
\theoremstyle{definition} 
\theoremstyle{definition} 
\theoremstyle{definition} \newtheorem{rmk}[cor]{Remark}
\theoremstyle{definition} \newtheorem*{rmka}{Remark}
\theoremstyle{definition} 
\theoremstyle{definition} 

\newcommand{\Gwr}[1]{G \wr #1}
\newcommand{\D}{\mathrel{\mathcal D}}
\newcommand{\N}{\mathbb N}
\newcommand{\C}{\mathbb C}
\newcommand{\Cplx}{\mathcal C}
\newcommand{\Cplxinv}{\mathcal C_{\textup{inv}}}
\newcommand{\T}{\mathcal T}
\newcommand{\Tinv}{\mathcal T_{\textup{inv}}}
\newcommand{\up}{\textup}
\newcommand{\Z}{\mathbb Z}

\newcommand{\Y}{\mathcal Y}
\newcommand{\X}{\mathcal X}
\newcommand{\ran}{\textup{ran}}
\newcommand{\dom}{\textup{dom}}
\newcommand{\rk}{\textup{rk}}
\newcommand{\ld}{\lfloor}
\newcommand{\rd}{\rfloor}
\newcommand{\IRR}{\textup{IRR}}
\DeclareMathOperator{\tr}{trace}

\newcommand{\meet}{\wedge}
\newcommand{\MAX}{\textup{MAX}}
\newcommand{\Rot}{\textup{Rot}}

\author{Martin E.\ Malandro}

\begin{document}


\title{Fourier inversion for finite inverse semigroups}
\address{Department of Mathematics and Statistics, Box 2206, Sam Houston State University, Huntsville, TX 77341-2206}
\email{malandro@shsu.edu}

\begin{abstract}This paper continues the study of Fourier transforms on finite inverse semigroups, with a focus on Fourier inversion theorems and FFTs for new classes of inverse semigroups. We begin by introducing four inverse semigroup generalizations of the Fourier inversion theorem for finite groups. Next, we describe a general approach to the construction of fast inverse Fourier transforms for finite inverse semigroups complementary to an approach to FFTs given in previous work. Finally, we give fast inverse Fourier transforms for the symmetric inverse monoid and its wreath product by arbitrary finite groups, 
as well as fast Fourier and inverse Fourier transforms for the planar rook monoid, the partial cyclic shift monoid, and the partial rotation monoid. 
\end{abstract}



\maketitle

\section{Introduction}
\label{SecIntro}

The theory of Fourier analysis on finite groups unifies the classical discrete Fourier transform (DFT) and Yates' analysis of factorial designs \cite{Yates}. The classical DFT is the Fourier transform on $\Z_n$, the cyclic group of order $n$, while the analysis of Yates is the Fourier transform on ${\Z_2}^k$.
Fast Fourier transforms (FFTs) and fast inverse Fourier transforms (FIFTs) have been developed for a wide variety of abelian and nonabelian groups---see, e.g., \cite{Baum,Bluestein,Clausen,CooleyTukey,Maslen,DanDiameters,DanWreath}. For applications, see, e.g., \cite{Brigham,PersiBook,Persi,DanWreath,DanApps}. 

Inverse semigroups are generalizations of groups which encode partial symmetries \cite{Lawson}. Every group is an inverse semigroup, but not conversely. In \cite{RookFFT,InvSemiFFT} we extended the theory of Fourier analysis on finite groups to finite inverse semigroups and developed explicit FFTs for the symmetric inverse monoid (also called the {\em rook monoid}) $R_n$ and its wreath product by arbitrary finite groups. In \cite{RookPRD} we developed an application of Fourier analysis on the rook monoid to the analysis of partially ranked datasets. While we proved a handful of Fourier inversion theorems for $R_n$ in \cite{RookFFT}, there has not yet been a treatment of Fourier inversion for arbitrary inverse semigroups, nor has there been a treatment of fast Fourier inversion for inverse semigroups.

This paper addresses 
these issues. First, we develop a sequence of Fourier inversion formulas valid for arbitrary finite inverse semigroups. Second, we give a framework for the construction of fast Fourier inversion algorithms for finite inverse semigroups similar to the framework for inverse semigroup FFTs developed in \cite{InvSemiFFT}. Third, we show how this framework together with Maslen's algorithm for fast Fourier inversion for the symmetric group \cite{Maslen}, Rockmore's algorithm for fast Fourier inversion for symmetric group wreath products \cite{DanWreath}, and the algorithm of Bj\"orklund et al.\ for the efficient computation of the M\"obius transform on lattices with few irreducibles \cite{Bjork} combine to yield fast Fourier inversion algorithms for the rook monoid and its wreath product by arbitrary finite groups. These algorithms are complementary to the FFTs for these monoids given in \cite{Bjork,InvSemiFFT}. Finally, we give fast Fourier and inverse Fourier transforms for three inverse monoids not previously considered, namely the planar rook monoid, the partial cyclic shift monoid, and the partial rotation monoid.

Let $S$ be a finite inverse semigroup. We associate $\C$-valued functions on $S$ with elements of the semigroup algebra $\C S$ by associating the delta functions of the elements of $S$ with the elements of the natural basis $\{s\}_{s\in S}$ of $\C S$. Specifically, if $f:S\rightarrow \C$, i.e., 
\[
f=\sum_{s\in S}f(s)\delta_s, 
\]
then $f$ corresponds to the element
$
\sum_{s\in S}f(s)s \in \C S.
$

If $f\in \C S$ is expressed in terms of the natural basis, then the {\em Fourier transform of $f$} is the re-expression of $f$ in terms of a {\em Fourier basis} of $\C S$. Unlike the natural basis, Fourier bases of $\C S$ are defined by symmetry conditions on $S$ (Definition \ref{DefFourierBasis}). If $S=\Z_n$, the Fourier transform of $f\in \C \Z_n$ is the usual DFT of $f$, and the Fourier basis for $\C \Z_n$ is the usual basis of exponential functions \cite{RookFFT,InvSemiFFT,RookPRD}. 

If $f\in \C S$ is expressed with respect to a particular Fourier basis, then the {\em inverse Fourier transform of $f$} is the re-expression of $f$ in terms of the natural basis of $\C S$. As with groups, FFTs and FIFTs for inverse semigroups give rise to efficient algorithms for computing the convolution of functions $f,g\in \C S$. Naive methods for computing Fourier transforms, inverse transforms, and convolutions each require $|S|^2$ operations, where an {\em operation} is defined to be a complex multiplication and a complex addition. Faster methods have been developed for a wide variety of groups and, to a lesser extent, non-group inverse semigroups. For example, it is known that the Fourier transform of $f\in \C S$ can be computed in no more than
\begin{itemize}
	\item $O(|S| \log |S|)$ operations if $S=\Z_n$ \cite{Bluestein},
	\item $O(|S| \log |S|)$ operations if $S$ is a supersolvable group \cite{Baum},
	\item $O(|S| \log^2 |S|)$ operations if $S=S_n$, the symmetric group on $n$ objects \cite{Maslen},
	\item $O(|S| \log^4 |S|)$ operations if $S=B_n$, the hyperoctahedral group on $n$ objects \cite{DanWreath}, and
	\item $O(|S| \log^2 |S|)$ operations if $S=R_n$, the rook monoid on $n$ objects \cite{Bjork,InvSemiFFT}.
\end{itemize}
Furthermore, if $f\in \C S$ is expressed with respect to particular computationally advantageous Fourier bases, then efficient algorithms for computing the inverse Fourier transform of $f$ exist, which require no more than
\begin{itemize}
	\item $O(|S|\log |S|)$ operations if $S=\Z_n$ \cite{Bluestein},
	\item $O(|S| \log |S|)$ operations if $S$ is a supersolvable group \cite{Baum,ClausenInversion},
	\item $O(|S|\log^2 |S|)$ operations if $S=S_n$ \cite{Maslen}, and
	\item $O(|S|\log^4 |S|)$ operations if $S=B_n$ \cite{DanWreath}.
\end{itemize}

We make the distinction in complexity between Fourier transforms and their inverses for the semigroups listed above because, while the classical FFT for $\Z_n$ automatically gives rise to an equally efficient algorithm for computing the inverse Fourier transform, it is not yet known whether FFTs automatically give rise to FIFTs for finite inverse semigroups in general. Given an FFT for a group, a nearly equally-efficient algorithm for computing the inverse Fourier transform arises by considering what is essentially the ``transpose" of the FFT \cite{ClausenInversion}. The key that enables this fast transpose algorithm is the Schur relations for group representations, which are not directly applicable to non-group inverse semigroups. 
Although we are able to show a relationship between the complexities of forward and inverse Fourier transforms for inverse semigroups in Theorems \ref{ThmFFTMainResult} and \ref{ThmMainInv}, we have not been able to establish a connection between Fourier transforms and their inverses for inverse semigroups as strong as the one for groups.

What about lower bounds? If $S$ is a group and all constants involved in multiplications in the computation of the Fourier transform are restricted to size no larger than 2, then it is known that the Fourier transform of an arbitrary $f\in \C S$ requires at least $\frac{1}{4}|S| \log |S|$ operations \cite{ClausenInversion}. This bound does not hold for inverse semigroups in general. In Remark \ref{LinearCplx} we point out an infinite family of inverse semigroups $S$ for which the Fourier transform of $f\in \C S$ can be computed in $|S|$ operations. However, the semigroups in this family are all idempotent and therefore have only trivial maximal subgroups. The question then becomes: Are there any interesting inverse semigroups $S$ with nontrivial maximal subgroups whose Fourier transform is sub-$O(|S|\log|S|)$ in complexity? As a first step towards answering this question, we exhibit in Section \ref{SecCnAndRotn} two inverse semigroup generalizations $S$ of $\Z_n$ for which a key step in computing the Fourier transform (which has been $O(|S|\log |S|)$ or worse in complexity for previously-considered inverse semigroups) can be completed in $O(|S| \log \log |S|)$ operations.

The main results of this paper are as follows. Let $S$ be a finite inverse semigroup with $\D$-classes $D_0 ,\ldots, D_n$, natural partial order $\leq$, and M\"obius function $\mu$. First, we have the following Fourier inversion formula.

\begin{thma}[Theorem \ref{ThmSemigpBasisInvThm}]
Let $f=\sum_{s\in S}f(s)s \in \C S$ and let $\X$ be any complete set of inequivalent, irreducible representations of $\C S$. For $t\in D_j$, let $G(t)$ denote the maximal subgroup at any idempotent of $D_j$ and let $r(t)$ denote the number of idempotents in $D_j$. Then for any $s\in S$, we have
\[
f(s)=\sum_{\rho\in{\X}}d_\rho 
\sum_{t\in S:t\geq s} \frac{\mu(s,t)}{r(t)|G(t)|}
\sum_{v\in S:v^{-1}\leq t^{-1}} \mu(v^{-1},t^{-1}) \tr \left(\hat f(\rho) \rho(v^{-1})\right).
\]
\end{thma}

Next, we have the following bound on the complexity of the inverse Fourier transform on $S$.

\begin{thma}[Theorem \ref{ThmMainInv}]For each $\D$-class $D_k$, pick an idempotent $e_k$, let $G_k$ be the maximal subgroup of $S$ at $e_k$, and let $\IRR(G_k)$ be a complete set of inequivalent irreducible representations of $\C G_k$. Let $\Y$ be the set of representations of $\C S$ induced by the $\IRR(G_k)$ (Definition \ref{DefInduced}). The number of operations required to compute the inverse Fourier transform of an arbitrary $\C$-valued function $f$ on $S$ expressed with respect to $\Y$ is no more than
\[
\Cplx(\mu_S) + \sum_{k=0}^n r_k^2 \Tinv(\IRR(G_k)),
\]
where $\Cplx(\mu_S)$ is the maximum number of operations required to compute the M\"obius transform of an arbitrary $\C$-valued function on $S$ and $\Tinv(\IRR(G_k))$ is the maximum number of operations required to compute the inverse Fourier transform of an arbitrary $\C$-valued function on $G_k$ expressed with respect to $\IRR(G_k)$.
\end{thma}

Finally, we give fast Fourier and inverse Fourier transforms for several families of inverse semigroups, which result in the following complexity bounds.

\begin{thma}[Theorem \ref{RnInversionResult}, Theorem \ref{ThmRnWreathInversion}, Theorem \ref{ThmPlanarRookFFT}, Theorem \ref{ThmFFTCn}, Theorem \ref{ThmRotnCplx}] There exists a complete set of inequivalent, irreducible representations $\Y$ of $\C S$ such that the Fourier transform and the inverse Fourier transform relative to $\Y$ of an arbitrary element $f\in \C S$ can be computed in no more than
\begin{itemize}
	\item $O(|S|\log^2 |S|)$ operations if $S=R_n$, the rook monoid on $n$ objects,
	\item $O(|S|\log^4 |S|)$ operations if $S=\Gwr R_n$, the wreath product of $R_n$ with any finite group $G$,
	\item $O(|S|\log^2 |S|)$ operations if $S=P_n$, the planar rook monoid on $n$ objects,
	\item $O(|S|\log^2 |S|)$ operations if $S=C_n$, the partial cyclic shift monoid on $n$ objects, and 
	\item $O(|S|\log |S|)$ operations if $S=\Rot_n$, the partial rotation monoid on $n$ objects.
\end{itemize} 
\end{thma}

We proceed as follows. Although we assume some familiarity with the ideas in \cite{InvSemiFFT}, we begin in Section \ref{SecBackground} with a quick review of the major ideas and terminology we need for our developments. Specifically, we review
some basic inverse semigroup theory (Section \ref{SubSecISGS}), the definition of the Fourier transform for finite inverse semigroups (Section \ref{SubSecFTDefs}), B.\ Steinberg's isomorphism between the inverse semigroup algebra $\C S$ and a direct sum of matrix algebras over group algebras \cite{Steinberg2} (Section \ref{SecMatrixAlgsOverGpAlgs}), and a general approach to the construction of FFTs for inverse semigroups (Section \ref{SubSecEvalFT}). 
In Section \ref{SecFourierInversionThm} we state and prove our four Fourier inversion formulas for finite inverse semigroups.
In Section \ref{SecFastInversion} we introduce a general method for the construction of fast inverse Fourier transforms on finite inverse semigroups and we establish a general bound on the inverse Fourier transform on a finite inverse semigroup relative to an induced set of representations. 
Section \ref{FastInversionSpecificISGS} contains our new fast Fourier and inverse Fourier transforms. We give fast inverse Fourier transforms for the rook monoid and its wreath product by arbitrary finite groups in Sections \ref{SecRookFIFT} and \ref{SecWreathFIFT}. Section \ref{SecPlanarRookFFT} contains our FFT and FIFT for the planar rook monoid. In Section \ref{SecCnAndRotn} we introduce the partial cyclic shift monoid and the partial rotation monoid, and we conclude with FFTs and FIFTs for these monoids in Sections \ref{SubsecPartialCyclicShift} and \ref{SubsecPartialRotation}.

\section{Background material}

\label{SecBackground}

\subsection{Inverse semigroups}

\label{SubSecISGS}

A semigroup is a nonempty set with an associative binary operation. A monoid is a semigroup with identity. Unless otherwise specified, we will write our semigroup operations multiplicatively. 

\begin{defn}A semigroup $S$ is an {\em inverse semigroup} if for each $x\in S$ there exists a unique $y\in S$ such that $xyx=x$ and $yxy=y$. In this case we say that $y$ is the {\em inverse} of $x$, and we write $x^{-1}=y$. 
\end{defn}

It follows that in an inverse semigroup, $xx^{-1}$ and $x^{-1}x$ are idempotent, and if $e$ is idempotent then $e=e^{-1}$. It is clear that every group is an inverse semigroup, and it is straightforward to show that an inverse semigroup is a group if and only if it has exactly one idempotent (the identity of the group). An {\em inverse monoid} is an inverse semigroup with an identity.



The {\em symmetric inverse monoid} (also known as the {\em rook monoid}) $R_n$ is the set of all injective partial functions from $\{1,2,\ldots,n\}$ to $\{1,2,\ldots,n\}$ (including the function with empty domain and range) under the usual operation of partial function composition. In this paper we view maps as acting on the left of sets and we compose maps from right to left, so if $\sigma,\gamma\in R_n$, then $\sigma \gamma\in R_n$ is the partial function whose domain is 
given by 
\[
\dom(\sigma \gamma) = \{x\in \{1,2,\ldots,n\}: x\in \dom(\gamma) \up{ and } \gamma(x)\in \dom(\sigma)\}, 
\]
and if $x\in \dom(\sigma\gamma),$ then $(\sigma\gamma)(x)=\sigma(\gamma(x))$.

\begin{defn}
If $S$ and $T$ are semigroups, then a {\em homomorphism} $\phi:S\rightarrow T$ is a map such that $\phi(ab)=\phi(a)\phi(b)$ for all $a,b\in S$. An {\em isomorphism} is a bijective homomorphism.
\end{defn}

A {\em rook matrix} of dimension $n$ is an $n\times n$ matrix with entries in $\{0,1\}$ that has at most one 1 in each row and column. Such a matrix can be thought of as a placement of non-attacking rooks on an $n\times n$ chessboard. $R_n$ is called the rook monoid because it is isomorphic to the semigroup of rook matrices of dimension $n$ under ordinary matrix multiplication, where the partial function $\sigma\in R_n$ corresponds to the rook matrix that has a 1 in the $i,j$ position whenever $\sigma(j)=i$, and a $0$ in all other positions \cite{Solomon}.
The {\em rank} of an element $\sigma\in R_n$ is $|\dom(\sigma)|=|\ran(\sigma)|$. 
The rook monoid plays the same role for inverse semigroups that the symmetric group does for groups, in the following variation of Cayley's theorem \cite{Lawson}.

\begin{thm}
\label{ThmCayleyInvSemi}
If $S$ is a finite inverse semigroup, then $S$ is isomorphic to an inverse sub-semigroup of $R_{|S|}$.
\end{thm}


It is easy to see that $R_n$ contains $\binom{n}{k}^2 k!$ elements of rank $k$, so we have:

\begin{thm}$|R_n|=\sum_{k=0}^n \binom{n}{k}^2 k!$.
\end{thm} 

\subsection{The Fourier transform}

\label{SubSecFTDefs}

Let $S$ be a finite inverse semigroup. The natural basis of the semigroup algebra $\C S$ is called the {\em semigroup basis}. Multiplication in $\C S$ (also called {\em convolution}) is given by the linear extension of the multiplication in $S$ via the distributive law.
%
As with groups, Fourier transforms are defined in terms of representations. 

\begin{defn}A {\em matrix representation} (or just {\em representation}) $\rho$ of $\C S$ is an algebra homomorphism $\rho:\C S\rightarrow M_{n}(\C)$ for some $n\in \N$. The number $n$ is the {\em dimension} of $\rho$, which we will denote by $d_\rho$.
\end{defn}



%

\begin{defn}
Let $f\in \C S$ with $f=\sum_{s\in S}f(s)s$. If $\rho$ is a representation of $\C S$, then the {\em Fourier transform of $f$ at $\rho$}, denoted $\hat{f}(\rho)$, is
\[
\hat{f}(\rho)=\rho(f) = \sum_{s\in S}f(s)\rho(s).
\]
\end{defn}

Adjectives such as {\em irreducible}, {\em inequivalent}, and {\em complete} apply to sets of representations of $\C S$ in a fashion similar to that for group algebras. Precise definitions may be found in \cite{InvSemiFFT,RookPRD}. Munn showed that $\C S$ is semisimple \cite{Munn},
so Wedderburn's theorem applies to $\C S$.

\begin{thm}[Wedderburn's theorem] Let $\X$ be a complete set of inequivalent, irreducible representations of $\C S$. Then $\X$ is finite, and $\X$ induces an algebra isomorphism (called the {\em Wedderburn isomorphism induced by $\X$}, or just the {\em Wedderburn isomorphism} if $\X$ is understood)
\begin{equation}
\label{eqWedderburn}
\bigoplus_{\rho\in\X}\rho: \C S\rightarrow \bigoplus_{\rho\in\X} M_{d_\rho}(\C).
\end{equation}
Explicitly, if $f\in \C S$ with $f=\sum_{s\in S}f(s)s$, then
\[
f\mapsto \bigoplus_{\rho\in\X}\left(\sum_{s\in S}f(s)\rho(s)\right)
=\bigoplus_{\rho\in\X} \hat f(\rho)
\]
in the Wedderburn isomorphism induced by $\X$. 
\end{thm}

Let $\X$ be any complete set of inequivalent irreducible representations of $\C S$. 

\begin{defn}
The Wedderburn isomorphism induced by $\X$ is called the {\em Fourier transform on $\C S$ relative to $\X$} (or just the {\em Fourier transform on $\C S$}, if $\X$ is understood). In particular, if $f\in \C S$ with $f=\sum_{s\in S}f(s)s$, then the {\em Fourier transform of $f\in \C S$ relative to $\X$} is the block matrix
\[
\bigoplus_{\rho\in\X}\left(\sum_{s\in S}f(s)\rho(s)\right) = 
\bigoplus_{\rho\in\X}\hat f (\rho).
\]
\end{defn}


The Fourier transform of $f$ relative to $\X$ can also be described in terms of a change of basis within $\C S$. The natural basis of the algebra on the right in (\ref{eqWedderburn}) is the set of matrices in this algebra with a 1 in one position and 0 in all other positions, and the inverse image of this basis in $\C S$ is the target basis in $\C S$ for the Fourier transform relative to $\X$. 

\begin{defn}
\label{DefFourierBasis}
If $\X$ is a complete set of inequivalent irreducible representations of $\C S$, then the inverse image of the natural basis of the algebra on the right in (\ref{eqWedderburn}) is called the {\em Fourier basis of $\C S$ relative to $\X$}. It is also called the {\em dual matrix coefficient basis for $\C S$ relative to $\X$} \cite{Maslen}.
\end{defn}

Thus, the Fourier transform of $f$ relative to $\X$ is a collection of (matrix) coefficients which provide the expression of $f$ in terms of the Fourier basis of $\C S$ relative to $\X$. 
In general, a {\em Fourier basis} of $\C S$ is any basis of $\C S$ which arises in this manner by some choice $\X$ of inequivalent, irreducible representations of $\C S$. If $f\in \C S$ is expressed with respect to the Fourier basis of $\C S$ relative to $\X$, we simply say that $f$ is expressed with respect to $\X$. Note that a dimensionality count of the algebras in (\ref{eqWedderburn}) yields
\begin{equation*}
|S|=\sum_{\rho\in\X}{d_\rho}^2.
\end{equation*}

For $S=\Z_n=\{0 ,1, \ldots, n-1\}$, the cyclic group of order $n$, the Fourier transform of $f=\sum_{t=0}^{n-1}f(t)t\in \C \Z_n$ is a diagonal matrix whose entries comprise the usual discrete Fourier transform of $f$ \cite{RookFFT,InvSemiFFT,RookPRD}.

\begin{defn}
If $f\in \C S$ is expressed with respect to $\X$, then the {\em inverse Fourier transform of $f$} is the collection of coefficients $\{f(s):s\in S\}$ such that $f=\sum_{s\in S}f(s)s$. The {\em inverse Fourier transform on $\C S$ relative to $\X$} is the change of basis within $\C S$ from the Fourier basis relative to $\X$ to the natural basis.
\end{defn}

Viewing Fourier transforms and inverse Fourier transforms as changes of basis within $\C S$, it follows that a naive implementation of the Fourier transform requires $|S|^2$ operations to apply to an arbitrary element of $\C S$, assuming that a complete set of inequivalent, irreducible representations of $\C S$ is precomputed on the natural basis of $\C S$ and stored in memory. Similarly, the inverse Fourier transform requires $|S|^2$ operations to apply to an arbitrary element of $\C S$, provided the inverse of the Fourier transform matrix is precomputed and stored in memory. A naive implementation of convolution also requires $|S|^2$ operations to compute the product $fg$ for  $f,g\in \C S$.

The convolution theorem for inverse semigroups is simply the restatement of the fact that the Wedderburn isomorphism is a homomorphism.

\begin{thm}The Fourier transform on $\C S$ turns convolution of functions into multiplication of block-diagonal matrices. The Fourier transform turns convolution into pointwise multiplication if and only if every irreducible representation of $\C S$ has dimension 1.
\end{thm}

As with groups, it follows that efficient methods for computing Fourier transforms and inverse Fourier transforms in $\C S$ relative to the {\em same} set of inequivalent, irreducible representations $\X$ of $\C S$ give rise to an efficient method for computing the convolution of functions on $S$. Specifically, to compute the convolution of $f,g\in\C S$, first compute the Fourier transforms $\hat f$ and $\hat g$ of $f$ and $g$, then form the matrix product $\hat f \hat g$, and finally compute the inverse Fourier transform of $\hat f \hat g$.

\subsection{Matrix algebras over group algebras}
\label{SecMatrixAlgsOverGpAlgs}

Let $S$ be a finite inverse semigroup. 
In this section we recall B.\ Steinberg's isomorphism between $\C S$ and a direct sum of matrix algebras over group algebras \cite{Steinberg2}, which extends ideas of Munn and Solomon \cite{Munn1,Munn3,Munn,Solomon}. We will need this result for our proofs of our Fourier inversion formulas in Section \ref{SecFourierInversionThm} and in our approach to fast Fourier inversion in Section \ref{SecFastInversion}. Before we can state his isomorphism we need to review four ideas---the natural partial order on $S$, the groupoid basis of $\C S$, the maximal subgroups of $S$, and Green's $\D$-relation. First we recall the natural partial order on $S$ \cite{Lawson}.

\begin{defn}For $s,t\in S$, we say $t\leq s$ if and only if there exists an idempotent $e\in S$ such that 
$
t=es.
$
\end{defn}
Whenever we refer to a partial order on $S$ we will always mean the natural partial order. 
The idempotents of $R_n$ are the restrictions of the identity mapping, so for $s,t\in R_n$, we have that $t\leq s$ if and only if $s$ extends $t$ as a partial function. If $S$ is a group, then the natural partial order on $S$ is just equality. 

Next we recall Steinberg's groupoid basis, which is the basis used to implement his isomorphism \cite{Steinberg2}. 

\begin{defn}For $s\in S$, define
\[
\ld s \rd = \sum_{t\in S:t\leq s} \mu(t,s)t \in \C S,
\]
where $\mu$ is the M\"obius function of the natural partial order on $S$. The collection $\{\ld s\rd:s\in S\}$ is called the {\em groupoid basis} of $\C S$.
\end{defn}

If $s,t\in R_n$ with $t\leq s$, then it is well known \cite{Stanley, Steinberg2} that 
$
\mu(t,s)=(-1)^{\rk(s)-\rk(t)}.
$

Of course, we may recover the natural basis of $\C S$ by M\"obius inversion---for $s\in S$, in $\C S$ we have
\[
s=\sum_{t\in S:t\leq s}\ld t \rd.
\]
It follows that if $f=\sum_{s\in S}f(s) \in \C S$, then writing $f$ with respect to the groupoid basis, we have $f=\sum_{s\in S}g(s)\ld s\rd$, where
\[
g(s)=\sum_{t\in S:t\geq s}f(t).
\]
Similarly, if $g=\sum_{s\in S}g(s)\ld s \rd \in \C S$, then writing $g$ with respect to the semigroup basis, we have $g=\sum_{s\in S}f(s)s$, where
\[
f(s)=\sum_{t\in S:t\geq s}\mu(s,t) g(t).
\]
We call these changes of basis the {\em zeta transform} and the {\em M\"obius transform} on $\C S$, respectively.

\begin{defn}The {\em zeta transform of a function $f:S\rightarrow \C$} is the collection of coefficients $\{ \sum_{t\in S:t\geq s}f(t):s\in S\}$. The {\em M\"obius transform of a function $f:S\rightarrow \C$} is the collection of coefficients $\{\sum_{t\in S:t\geq s} \mu(s,t) f(s) :s\in S\}$.
\end{defn}
The groupoid basis multiplies in the following manner \cite{Steinberg2}. 
\begin{thm}\label{ThmGroupoidBasisMult}
For $s,t\in S$, in $\C S$ we have
\[
\ld s \rd \ld t \rd = 
\begin{cases}
\ld st \rd &\up{if }s^{-1}s=tt^{-1};\\
0 &\up{otherwise}.
\end{cases}
\]
\end{thm}

To motivate the importance of the groupoid basis, note that there is an alternative model for the composition of partial functions: for $s,t\in R_n$, we could allow the composition $s t$ if and only if the domain of $s$ lines up exactly with the range of $t$. For $s\in R_n$, $s^{-1}s$ is the partial identity on $\dom(s)$ and $ss^{-1}$ is the partial identity on $\ran(s)$, so it follows that for $s,t\in R_n$, in $\C R_n$ we have
\[
\ld s \rd \ld t \rd = 
\begin{cases}
\ld st \rd &\up{if }\dom(s)=\ran(t);\\
0 &\up{otherwise}.
\end{cases}
\]
That is, the multiplication of the groupoid basis of $\C R_n$ encodes this alternate model of partial function composition \cite{Steinberg2}. 
If $A\subseteq \{1,2,\ldots,n\}$, we will identify $A$ with the partial identity on $A$ in $R_n$, so for $s\in R_n$, we will write $s^{-1}s=\dom(s)$ and $ss^{-1}=\ran(s)$. In fact, for any finite inverse semigroup $S$ and $s\in S$, it is customary to write
\begin{equation}
\label{EqDomRanS}
\begin{aligned}
&\dom(s)=s^{-1}s,\\
&\ran(s)=ss^{-1}.
\end{aligned}
\end{equation}
The reason for this is that $S$ is isomorphic to an inverse sub-semigroup of $R_{|S|}$, and if we identify $S$ with an embedding of $S$ in $R_{|S|}$, then for $s\in S$ we have $s^{-1}s=\dom(s)$ and $ss^{-1}=\ran(s)$. On the other hand, the notation in \eqref{EqDomRanS} makes sense even without an embedding of $S$ into $R_{|S|}$. When $S$ is an arbitrary finite inverse semigroup and we make use of this notation, we do so without reference to any particular embedding of $S$ in any rook monoid. Note that, for any $s\in S$, we have $\dom(s^{-1})=\ran(s)$ and $\ran(s^{-1})=\dom(s)$, and if $e\in S$ is idempotent, then $\ran(e)=\dom(e)=e$. 

%

\begin{defn}A subset $G$ of $S$ is said to be a {\em subgroup} of $S$ if $G$ is a group under the operation of $S$. If $G$ is a subgroup of $S$, then $G$ is said to be a {\em maximal subgroup} if $G$ is not properly contained in any other subgroup of $S$.
\end{defn}

Each idempotent $e$ of $S$ is the identity for a unique maximal subgroup of $S$, called the {\em maximal subgroup of $S$ at $e$} \cite{CliffPres}. In fact, denoting the maximal subgroup of $S$ at $e$ by $G_e$, we have \cite{Steinberg2}
$$
G_e = \{s\in S:\dom(s)=\ran(s)=e \}  .
$$
Thus if $e\in R_n$ is idempotent and $\rk(e)=k$, then $G_e$ is isomorphic to $S_k$, the permutation group on $k$ elements.

Finally, we recall Green's $\D$-relation \cite{Green}.

\begin{defn}For $s,t\in S$, we say that $s$ and $t$ are {\em $\D$-related} and we write $s\D t$ if there exists $x\in S$ such that $\dom(x)=\ran(s)$ and $\ran(x)=\ran(t)$.
\end{defn}
$\D$ is an equivalence relation on $S$, and the equivalence classes of $S$ under $\D$ are called the $\D$-classes of $S$. 
We note that if $s,t\in S$ and $\dom(s)=\ran(t)$, then we certainly have that $s\D t$ (by taking $x=s^{-1}$). For $s,t \in R_n$, we have that $s\D t$ if and only if $\rk(s)=\rk(t)$, so $R_n$ has $n+1$ $\D$-classes. Indeed, for $s,t\in R_n$, if $\rk(s)\neq \rk(t)$ then $s$ and $t$ are certainly not $\D$-related. On the other hand, if $\rk(s)=\rk(t)$, then taking $x\in R_n$ to be the unique order-preserving bijection from $\ran(s)$ to $\ran(t)$ shows $s\D t$.

We now describe Steinberg's isomorphism. Let $D_0,\ldots,D_n$ denote the $\D$-classes of $S$, and let $\C D_k$ denote the $\C$-span of $\{\ld s\rd:s\in D_k\}$. By Theorem \ref{ThmGroupoidBasisMult}, as an algebra $\C S=\bigoplus_{k=0}^n \C D_k$. For each $\D$-class $D_k$, fix an idempotent $e_k$, and let $G_k$ denote the maximal subgroup of $S$ at $e_k$. For each idempotent $e\in D_k$, fix an element $p_e\in S$ such that $\dom(p_e)=e_k$ and $\ran(p_e)=e$, taking $p_{e_k}=e_k$. Let $r_k$ denote the number of idempotents in $D_k$. Steinberg gives the following explicit algebra isomorphism from $\C D_k$ to $M_{r_k}(\C G_k)$ \cite{Steinberg2}. 

\begin{thm}\label{ThmSteinbergIsom}
Viewing $r_k\times r_k$ matrices as being indexed by pairs of idempotents in $D_k$, define a map $\Phi:\C D_k\rightarrow M_{r_k}(\C G_k)$ by, for $s\in D_k$,
\[
\Phi(\ld s\rd)={p_{\ran(s)}}^{-1} s p_{\dom(s)} E_{\ran(s),\dom(s)},
\]
where $E_{\ran(s),\dom(s)}$ is the standard $r_k\times r_k$ matrix with a 1 in the $\ran(s),\dom(s)$ position and $0$ elsewhere, and extending linearly to the rest of $\C D_k$.  Then $\Phi$ is an isomorphism, with inverse induced by, for $s\in G_k$,
\[
sE_{e,f} \mapsto \ld p_e s p_f^{-1}\rd.
\]
\end{thm}
Note that $p_e\in D_k$ implies ${p_e}^{-1}\in D_k$, and note that if $s\in D_k$ then ${p_{\ran(s)}}^{-1} s p_{\dom(s)}\in G_k$ by construction.
Since $\C S=\bigoplus_{k=0}^n\C D_k$, it follows that
\[
\C S \cong \bigoplus_{k=0}^n M_{r_k}(\C G_k).
\]
As a consequence of Theorem \ref{ThmSteinbergIsom}, we have the following method for generating the irreducible representations of $\C S$ from the irreducible representations of the maximal subgroups of $S$ \cite{Steinberg2}.

\begin{thm}\label{ThmRepInducing} Let $\IRR(G_k)$ be a complete set of inequivalent irreducible representations of $\C G_k$. If $\rho\in \IRR(G_k)$, define the representation $\bar\rho$ of $\C S$ in the following way. First, define $\bar\rho$ on $M_{r_k}(\C G_k)$ by, for $g\in G_k$ and idempotents $a,b\in D_k$,
\[
\bar\rho(gE_{a,b}) = E_{a,b}\otimes \rho(g),
\] 
and extending linearly to the rest of $M_{r_k}(\C G_k)$. Then extend $\bar\rho$ to the rest of $\bigoplus_{k=0}^n M_{r_k}(\C G_k)$ (and hence to $\C S$) by letting $\bar\rho$ be 0 on the other summands. Then $\Y=\{\bar\rho: \rho \in \biguplus_{k=0}^n \IRR(G_k)\}$ is a complete set of inequivalent irreducible representations of $\C S$. 
\end{thm}

\begin{defn}
\label{DefInduced}
With notation as in Theorem \ref{ThmRepInducing}, for $\rho \in \IRR(G_k)$, let $\bar\rho$ be the corresponding irreducible representation of $\C S$. We call $\Y=\{\bar \rho:\rho \in \biguplus_{k=0}^n \IRR(G_k)\}$ an {\em induced} set of representations of $\C S$.
\end{defn}

By Theorem \ref{ThmRepInducing}, an induced set of representations of $\C S$ is automatically a complete set of inequivalent, irreducible representations of $\C S$. Throughout the paper we reserve the notation $\Y$ to refer to an induced set of representations of $\C S$, while using the notation $\X$ to refer to an arbitrary complete set of inequivalent, irreducible representations of $\C S$.

\subsection{Evaluating the Fourier transform}
\label{SubSecEvalFT}

Let $S$ be a finite inverse semigroup. 
Let $D_0,\ldots,D_n$ be the $\D$-classes of $S$, let $r_k$ denote the number of idempotents in $D_k$, and choose an idempotent $e_k$ from each $\D$-class $D_k$. For every idempotent $e\in D_k$, fix an element $p_e\in S$ such that $\dom(p_e)=e_k$ and $\ran(p_e)=e$, taking $p_{e_k}=e_k$. Let $G_k$ be the maximal subgroup at $e_k$ and let $\IRR(G_k)$ be a complete set of inequivalent irreducible representations of $\C G_k$. With notation as in Theorem \ref{ThmRepInducing}, let $\Y$ be the induced set of representations of $\C S$ given by $\Y=\{\bar \rho:\rho \in \biguplus_{k=0}^n \IRR(G_k)\}$.

We now recall the main idea from \cite{InvSemiFFT}, which describes the structure of the Fourier transform on $\C S$ relative to $\Y$. Let $f=\sum_{s\in S}f(s)s \in \C S.$ Writing $f$ relative to the groupoid basis of $\C S$, we have
\[
f=\sum_{s\in S}g(s)\ld s \rd,
\]
where $g:S\rightarrow \C$ is the function given by
\[
g(s)=\sum_{t\in S: t\geq s}f(t).
\]
If $\rho\in \IRR(G_k)$, then 
\[
\bar\rho(\ld s \rd)= 
\begin{cases}
E_{\ran(s),\dom(s)}\otimes\rho({p_{\ran(s)}}^{-1} s p_{\dom(s)} ) & \up{if }s\in D_k;\\
0& \up{otherwise.} 
\end{cases}
\]
View $\hat f(\bar \rho)$ as an $r_k \times r_k$ matrix whose rows and columns are indexed by the idempotents in $D_k$ and whose entries are themselves $d_\rho \times d_\rho$ matrices. By Theorem \ref{ThmSteinbergIsom}, for idempotents $a,b\in D_k$ we have
\begin{align*}
\hat f(\bar \rho)_{a,b} &= \sum_{s\in D_k : \ran(s)=a, \dom(s)=b} g(s)\rho({p_{a}}^{-1} s p_b  )\\
&=\sum_{s\in G_k}g(p_a s {p_b}^{-1})\rho(s).
\end{align*}
If we define a function $h_{a,b}:G_k\rightarrow \C$ by, for $s\in G_k$,
\[
h_{a,b}(s)=g({p_a} s p_b^{-1}),
\]
then we see that
\[
\hat f(\bar \rho)_{a,b} = \sum_{s\in G_k} h_{a,b}(s)\rho(s),
\]
the Fourier transform (in $\C G_k$) of $h_{a,b}$ at $\rho$.

Thus, the $a,b$ entry of $\hat f(\bar \rho)$ is a function of the coefficients $$\{g(s):s\in D_k, \ran(s)=a,\dom(s)=b\}.$$ In light of this, in \cite{RookFFT} we proposed the following general approach to the construction of FFTs for $\C S$: To compute the Fourier transform of $f=\sum_{s\in S}f(s)s \in \C S$ relative to an induced set of representations of $\C S$, first compute the change of basis of $f$ to the groupoid basis (that is, compute the zeta transform of $f$), and then for each $\D$-class $D_k$, compute $r_k^2$ group Fourier transforms on $G_k$. 
This gave the following result \cite{InvSemiFFT}.
\begin{thm}
\label{ThmFFTMainResult}
The number of operations required to compute the Fourier transform of an arbitrary $\C$-valued function $f$ on $S$ is no more than
\[
\Cplx(\zeta_S) + \sum_{k=0}^n r_k^2 \T(\IRR(G_k)),
\]
where $\Cplx(\zeta_S)$ is the maximum number of operations required to compute the zeta transform of an arbitrary $\C$-valued function on $S$ and $\T(\IRR(G_k))$ is the maximum number of operations required to compute the Fourier transform of an arbitrary $\C$-valued function on $G_k$ relative to $\IRR(G_k)$.
\end{thm}

\section{Fourier inversion formulas for finite inverse semigroups}
\label{SecFourierInversionThm}

In this section we give a series of Fourier inversion theorems for arbitrary finite inverse semigroups (Theorems \ref{FInvThm}--\ref{ThmSemigpBasisInvThm}). Theorems \ref{FInvThm} and \ref{FInvThm2} are generalizations of Fourier inversion theorems proved for the rook monoid in \cite{RookFFT}, while Theorems \ref{FInvThm3} and \ref{ThmSemigpBasisInvThm} are new. We begin by recalling the Fourier inversion theorem for finite groups \cite[Section 6.2]{Serre}.

\begin{thm}
\label{FInvGrp}
Let $G$ be a finite group, and let $f=\sum_{s\in G}f(s)s \in \C G$. Let $\IRR(G)$ be a complete set of inequivalent, irreducible matrix representations of $\C G$. Then
\[
f(s) = \frac{1}{|G|} \sum_{\rho \in \IRR(G)} d_\rho \tr \left( \hat f(\rho)\rho(s^{-1}) \right).
\]
\end{thm}

Now, let $S$ be a finite inverse semigroup and let notation be as in Section \ref{SubSecEvalFT}.
Here is our first inversion theorem, which expresses Fourier inversion relative to $\Y$ in terms of the groupoid basis.

\begin{thm}
\label{FInvThm}
Let $g = \sum_{s\in S}g(s)\ld s \rd \in \C S$, and let $s\in D_k$. Let $y\in G_k$ be the element defined by
\[
y={p_{\ran(s)}}^{-1}sp_{\dom(s)}.
\]
For $\bar\rho \in \Y$, view $\hat g(\bar\rho)$ as an $r_k\times r_k$ matrix whose rows and columns are indexed by the idempotents in $D_k$ and whose entries are themselves $d_\rho \times d_\rho$ matrices. For idempotents $a,b\in D_k$, denote the $a,b$ entry of $\hat g(\bar\rho)$ (itself a $d_\rho \times d_\rho$ matrix) by $\hat g(\bar \rho)_{a,b}$. Then
\[
g(s) = \frac{1}{|G_k|} \sum_{\rho \in \IRR(G_k)} d_\rho \tr
\left(
\hat g(\bar \rho)_{\ran(s),\dom(s)} \rho(y^{-1})
\right).
\]
\end{thm}

\begin{proof} 
For $\rho \in \IRR(G_k)$ we have 
\[
\hat g(\bar \rho) = \sum_{x\in S}g(x)\bar\rho(\ld x \rd),
\]
with $\bar \rho(\ld x\rd) = 0$ if $x \notin D_k$. As in Section \ref{SubSecEvalFT}, the $\ran(s),\dom(s)$ entry of $\hat g(\bar \rho)$ is determined by the values $g(x)$ for which $\dom(x)=\dom(s)$ and $\ran(x)=\ran(s)$ (and the values $g(x)$ for such $x$ do not affect any of the other entries of $\hat g(\bar \rho)$), and if we define a function $g_{\ran(s),\dom(s)}: G_k\rightarrow \C$ by
\[
g_{\ran(s),\dom(s)}(x) = g(p_{\ran(s)} x {p_{\dom(s)}}^{-1}),
\]
then 
\begin{align*}
\hat g(\bar\rho)_{\ran(s),\dom(s)} &= \sum_{x\in G_k} g(p_{\ran(s)} x {p_{\dom(s)}}^{-1}) \rho(x) \\
&=\sum_{x\in G_k}g_{\ran(s),\dom(s)}(x)\rho(x)\\
&=\hat g_{\ran(s),\dom(s)} (\rho),
\end{align*}
the Fourier transform of $g_{\ran(s),\dom(s)}$ at $\rho$ in $\C G$.

Note that $s=p_{\ran(s)} y {p_{\dom(s)}}^{-1},$ because
\begin{align*}
s &= ss^{-1}ss^{-1}s\\
&= \ran(s)  s  \dom(s)\\
&= p_{\ran(s)}{p_{\ran(s)}}^{-1} s p_{\dom(s)} {p_{\dom(s)}}^{-1}\\
&= p_{\ran(s)} y {p_{\dom(s)}}^{-1}.
\end{align*}

The Fourier inversion theorem for groups applies to $g_{\ran(s),\dom(s)}$, and yields 
\begin{align*}
g(s) &= 
g(p_{\ran(s)} y {p_{\dom(s)}}^{-1}) \\ &= g_{\ran(s),\dom(s)}(y) \\&= 
\frac{1}{|G_k|} \sum_{\rho \in \IRR(G_k)} d_\rho \tr \left( \hat{g}_{\ran(s),\dom(s)}(\rho) \rho(y^{-1}) \right),
\end{align*}
and since
\[
\hat{g}_{\ran(s),\dom(s)}(\rho) =
\hat g(\bar \rho)_{\ran(s),\dom(s)}
,\]
we are done.
\end{proof}

Now, let ${\X}$ be any set of inequivalent, irreducible matrix representations for $\C S$. Here is our next inversion theorem, which expresses Fourier inversion relative to ${\X}$ in terms of the groupoid basis. 

\begin{thm}
\label{FInvThm2}
Let $g = \sum_{s\in S}g(s)\ld s \rd \in \C S.$
Let $s\in D_k$. For $\rho \in \IRR(G_k)$, let $\bar\rho\in \Y$ denote the corresponding induced representation of $\C S$, which is equivalent to some representation $\widetilde \rho \in \X$. 
Then
\[
g(s) = \frac{1}{|G_k|} \sum_{\rho \in \IRR(G_k)}d_\rho \tr
\left(
\hat g (\widetilde\rho) \widetilde\rho(\ld s^{-1}\rd)
\right).
\]
\end{thm}

\begin{proof}
Since $\widetilde\rho$ is equivalent to $\bar \rho$, we have $\bar \rho = A^{-1} \widetilde\rho A$
for some invertible matrix $A$. It follows that
\[
\hat g(\bar \rho) = A^{-1}\hat g(\widetilde\rho) A.
\]
As in Theorem \ref{FInvThm}, let $y\in G_k$ be the element defined by 
$y={p_{\ran(s)}}^{-1}sp_{\dom(s)}.$ Then
\begin{align*}
\tr\left(
\hat g(\bar \rho)_{\ran(s),\dom(s)} \rho(y^{-1})
\right)
=&
\tr \left(
\hat g(\bar\rho) \left(E_{\dom(s),\ran(s)}\otimes \rho(y^{-1})\right)
\right) \\ 
=&
\tr \left(
\hat g(\bar \rho)\bar\rho(\ld s^{-1}\rd)
\right) \\ 
=&
\tr \left(\left(
A^{-1}\hat g(\widetilde\rho) A\right) \left( A^{-1} \widetilde\rho(\ld s^{-1}\rd)A\right)
\right) \\ 
=&
\tr \left(
A^{-1}\hat g(\widetilde\rho)\widetilde\rho(\ld s^{-1}\rd)A 
\right) \\ 
=&
\tr \left(
\hat g(\widetilde\rho)\widetilde\rho(\ld s^{-1}\rd)
\right),
\end{align*}
where the last equality follows from the similarity-invariance of trace. The theorem now follows from Theorem \ref{FInvThm}.
\end{proof}

Our next inversion theorem also expresses Fourier inversion relative to $\X$ in terms of the groupoid basis, but does so without reference to the $\IRR(G_k)$.

\begin{thm}
\label{FInvThm3}
Let 
$
g = \sum_{s\in S}g(s)\ld s \rd \in \C S.
$
Let $s\in D_k$. Then
\[
g(s) = \frac{1}{r_k|G_k|} \sum_{\rho \in \X} d_{\rho} \tr
\left(
\hat g (\rho) \rho(\ld s^{-1}\rd)
\right).
\]
\end{thm}
\begin{proof}
If $\rho\in \X$, then $\rho$ is equivalent to some representation $\bar\rho\in\Y$, which was induced by some representation $\rho'\in \IRR(G_j)$ for some $j\in \{0,1,\ldots,n\}$.

Notice that since $s\in D_k$, we also have $s^{-1}\in D_k$, and thus $\bar\rho(\ld s^{-1} \rd)$ is 0 unless $\rho' \in \IRR(G_k)$. It follows that $\rho(\ld s^{-1} \rd)$ is 0 unless $\rho' \in \IRR(G_k)$. If $\rho' \in \IRR(G_k)$, then we have $d_{\rho} = d_{\bar\rho} = r_k d_{\rho'}$, so $d_{\rho'}=d_{\rho}/r_k$. The theorem now follows from Theorem \ref{FInvThm2}.
\end{proof}

Finally, our last inversion theorem expresses Fourier inversion relative to $\X$ in terms of the semigroup basis.

\begin{thm}
\label{ThmSemigpBasisInvThm}
Let $f=\sum_{s\in S}f(s)s \in \C S$. For $t\in D_j$, let $G(t)$ denote $G_j$ (the maximal subgroup at $e_j$), and let $r(t)$ denote $r_j$ (the number of idempotents in $D_j$). Then for any $s\in S$ we have
\[
f(s)
=\sum_{\rho\in \X} d_\rho  
\sum_{t\in S:t\geq s} \frac{\mu(s,t)}{r(t)|G(t)|}
\sum_{v\in S:v^{-1}\leq t^{-1}} \mu(v^{-1},t^{-1}) \tr\left( \hat f(\rho) \rho(v^{-1})  \right).
\]
\end{thm}

\begin{proof}
We have that $f=\sum_{s\in S}g(s)\ld s \rd$, where $g(s)=\sum_{t\in S: t\geq s}f(t)$. For $s\in D_j$, by Theorem \ref{FInvThm3} we have
\begin{align*}
g(s) &= \frac{1}{r_j |G_j|} \sum_{\rho\in\X} d_{\rho} \tr\left( \hat f(\rho) \rho( \ld s^{-1} \rd)\right) \\
&= \frac{1}{r_j|G_j|} \sum_{\rho\in \X} d_{\rho} \tr \left( 
\hat f(\rho)\rho\left(
\sum_{t\in S:t^{-1}\leq s^{-1}} \mu(t^{-1},s^{-1}) t^{-1}
\right)\right) \\
&= \frac{1}{r_j|G_j|} \sum_{\rho\in \X} d_{\rho} \tr \left( 
\sum_{t\in S:t^{-1}\leq s^{-1}}\mu(t^{-1},s^{-1})
\hat f(\rho)\rho(t^{-1})
\right)\\
&= \frac{1}{r_j|G_j|} \sum_{\rho\in \X} d_{\rho} 
\sum_{t\in S:t^{-1}\leq s^{-1}}\mu(t^{-1},s^{-1})\tr \left(\hat f(\rho)\rho(t^{-1})\right).
\end{align*}
Since $g(s)=\sum_{t\in S:t\geq s}f(t)$, we have 
\begin{align*}
f(s)&=\sum_{t\in S:t\geq s}\mu(s,t)g(t)\\
&=\sum_{t\in S:t\geq s}\mu(s,t) \left(
\frac{1}{r(t)|G(t)|} \sum_{\rho\in \X} d_{\rho} 
\sum_{v\in S:v^{-1}\leq t^{-1}}\mu(v^{-1},t^{-1})\tr \left(\hat f(\rho)\rho(v^{-1})\right)
\right) \\
&=\sum_{\rho\in \X} d_\rho  
\sum_{t\in S:t\geq s} \frac{\mu(s,t)}{r(t)|G(t)|}
\sum_{v\in S:v^{-1}\leq t^{-1}} \mu(v^{-1},t^{-1}) \tr\left( \hat f(\rho) \rho(v^{-1})  \right).
\end{align*}
\end{proof}


\begin{rmka} 
If $S$ is a group, then the statements of Theorems \ref{FInvThm}--\ref{ThmSemigpBasisInvThm} all reduce to the statement of the Fourier inversion theorem for groups (Theorem \ref{FInvGrp}).
\end{rmka}


\section{Fast Fourier inversion for finite inverse semigroups---a unified approach}
\label{SecFastInversion}

In this section we describe a general method for the construction of fast inverse Fourier transforms for finite inverse semigroups, which results in general bounds on the complexity of the inverse Fourier transform in Theorem \ref{ThmMainInv} and Corollary \ref{CorMainInv}. In preparation for our results in Section \ref{FastInversionSpecificISGS}, we also explain how the results of Bj\"orklund et al.\ \cite{Bjork} can be used to bound certain terms appearing in Theorem \ref{ThmFFTMainResult}, Theorem \ref{ThmMainInv}, and Corollary \ref{CorMainInv}. Let $S$ be a finite inverse semigroup and let notation be as in Section \ref{SubSecEvalFT}.


\subsection{Designing a fast inverse Fourier transform}
\label{SubSecFastInversion} 

%
%

We begin with the following result on the complexity of the inverse Fourier transform. This result can be seen as the natural complement to Theorem \ref{ThmFFTMainResult}.

\begin{thm}
\label{ThmMainInv}If $\Y$ is an induced set of representations of $\C S$, then the number of operations required to compute the inverse Fourier transform of an arbitrary $\C$-valued function $f$ on $S$ expressed with respect to $\Y$ is no more than 
\[
\Cplx(\mu_S) + \sum_{k=0}^n r_k^2 \Tinv(\IRR(G_k)) ,
\]
where $\Cplx(\mu_S)$ is the maximum number of operations required to compute the M\"obius transform of an arbitrary $\C$-valued function on $S$ and $\Tinv(\IRR(G_k))$ is the maximum number of operations required to compute the inverse Fourier transform of an arbitrary $\C$-valued function on $G_k$ expressed with respect to $\IRR(G_k)$.
\end{thm}

\begin{proof}Given a set of induced representations $\Y$ of $\C S$ and an element $f\in \C S$ expressed with respect to $\Y$, we may find the coefficients $f(s)$ such that $f=\sum_{s\in S}f(s)s$ by first computing the coefficients $g(s)$ for which $f=\sum_{s\in S}g(s)\ld s \rd$, and then computing the coefficients $f(s)$ from the $g(s)$ by computing the change of basis from the groupoid basis to the semigroup basis.

To show that this approach results in the above bound, let $f=\sum_{s\in S}f(s)s$ be an arbitrary $\C$-valued function on $S$. We assume we have the coefficients of the matrices $\hat f(\bar\rho) = \bar\rho(f)$ for all $\bar\rho\in \Y$ stored in memory. Let $g:S\rightarrow \C$ be the function given by, for $s\in S$, $g(s)=\sum_{t\in S:t\geq s}f(t)$, so that
\[
f = \sum_{s\in S}f(s)s = \sum_{s\in S}g(s)\ld s \rd.
\]


Let $a,b\in D_k$ be idempotent and let $S_{a,b}=\{s\in D_k : \ran(s)=a, \dom(s)=b\}$.  
By Theorem \ref{ThmSteinbergIsom} we have the bijection
\[
h_{a,b}:S_{a,b}\rightarrow G_k
\]
given by $h_{a,b}(s)={p_a}^{-1}s p_b$, with inverse
\[
{h_{a,b}}^{-1}:G_k\rightarrow S_{a,b}
\]
given by ${h_{a,b}}^{-1}(s)=p_as{p_b}^{-1}$. Let $g_{a,b}:G_k\rightarrow \C$ by $g_{a,b}(s)=g({p_a} s {p_b}^{-1})$, so that for all $\rho\in \IRR(G_k)$, we have $\hat f(\bar \rho)_{a,b}=\hat g_{a,b}(\rho)$, the Fourier transform of $g_{a,b}$ at $\rho$ in $G_k$. If we invert the $\hat g_{a,b}(\rho)$ as $\rho$ varies over $\IRR(G_k)$, then we recover the coefficients $\{g_{a,b}(s):s\in G_k\}=\{g(p_as{p_b}^{-1}):s\in G_k\}=\{g(s):s\in S_{a,b}\}$. By assumption this computation requires no more than $\Tinv(\IRR(G_k))$ operations. Thus, computing the coefficients $g(s)$ for all $s\in S$ requires no more than
\[
\sum_{k=0}^n r_k^2 \Tinv(\IRR(G_k)) 
\]
operations. 

We can then compute the coefficients $f(s)$ from the coefficients $g(s)$ by computing the change of basis from the groupoid basis of $S$ to the semigroup basis (that is, by computing the M\"obius transform of the function $g$), which by assumption requires no more than $\Cplx(\mu_S)$ operations.
\end{proof}

\begin{cor}
\label{CorMainInv}Let $\Cplxinv(G_k)$ denote the quantity
\[
\Cplxinv(G_k) = \min\{\Tinv(\mathcal R_k)\},
\]
where the minimum is taken across all complete sets $\mathcal R_k$ of inequivalent, irreducible representations of $\C G_k$. Then there exists a complete set of inequivalent, irreducible representations ${\mathcal R}$ of $\C S$ such that the number of operations required to compute the inverse Fourier transform of an arbitrary $\C$-valued function on $S$ expressed with respect to ${\mathcal R}$ is no more than
\[
\Cplx(\mu_S) + \sum_{k=0}^n r_k^2 \Cplxinv(G_k).
\]
\end{cor}

\begin{proof} For $k\in \{0,\ldots,n\}$, let $\mathcal R_k$ be a complete set of inequivalent irreducible representations of $\C G_k$ such that the inverse Fourier transform of an arbitrary $\C$-valued function on $G_k$ expressed with respect to $\mathcal R_k$ can be computed in no more than $\Cplxinv(G_k)$ operations. Then let $ \mathcal R $ be the complete set of inequivalent irreducible representations of $\C S$ induced by $\biguplus_{k=0}^n \mathcal R_k$. The result follows from Theorem \ref{ThmMainInv}, with $\IRR(G_k) = \mathcal R_k$.
\end{proof}

\subsection{Fast zeta and M\"obius transforms}

In \cite{InvSemiFFT} we designed fast Fourier transforms for specific inverse semigroups of interest using Theorem \ref{ThmFFTMainResult}, by designing explicit fast zeta transforms for these inverse semigroups and combining them with existing algorithms for Fourier transforms on their maximal subgroups. The recent work of Bj\"orklund et al.\ \cite{Bjork} contains an algorithm that constructs small circuits for computing zeta and M\"obius transforms on finite lattices with few join-irreducibles. Recall that, if $L$ is a finite lattice, then an element $j\in L$ is {\em join-irreducible} if and only if $j$ covers exactly one other element of $L$. If $M$ is a finite meet-semilattice, denote by $L(M)$ the lattice obtained by adjoining a formal maximal element $\MAX$ to $M$ if $M$ is not already a lattice. If $M$ is a lattice, then let $L(M)=M$. The algorithm in \cite{Bjork} is easily modified to work for semilattices, which results in the following theorem.

\begin{thm}
\label{ThmBjork}
Let $(M,\leq)$ be a finite meet-semilattice with M\"obius function $\mu$ and suppose $L(M)$ has $v$ join-irreducible elements. Let $f: M \rightarrow \C$. For $s\in M$, let
\[
f_\zeta(s) = \sum_{t\in M: t\geq s}f(t)
\]
and
\[
f_\mu(s) = \sum_{t\in M: t\geq s}\mu(s,t) f(t).
\]
Then the collections of coefficients $\{ f_\zeta(s):s\in M \}$ and $\{ f_\mu(s):s\in M \}$ can each be computed in $O(|M|v)$ operations. 
\end{thm}

\begin{proof}The result is immediate from Theorem 1.1 of \cite{Bjork} if $M$ is a lattice, so suppose $M$ is not a lattice. Let $L=L(M)$, and denote by $\leq_L$ and $\mu_L$ the partial order and the M\"obius function of $L$, respectively. The algorithm in \cite{Bjork} finds arithmetic circuits, each of size $O(|L|v)$, for computing the upward M\"obius and zeta transforms of arbitrary $\C$-valued functions on $L$. (In our language, these circuits are algorithms which each require $O(|L|v)$ operations to run on an arbitrary $\C$-valued function on $L$ as input.) Let $f_L:L\rightarrow \C$ by 
\[
f_L(s) = 
\begin{cases}
f(s) & \up{if }s\in M;\\
0 & \up{if }s=\MAX,
\end{cases}
\]
so we can compute the coefficients 
\[
\zeta_L(s) = \sum_{t\in L: t \geq_L s} f_L(t)
\]
for all $s\in L$ and
\[
\mu_L(s) = \sum_{t\in L:t \geq_L s} \mu_L(s,t) f_L(t)
\]
for all $s\in L$, each in $O(|L|v)$ operations. For $s\in M$ we have
\begin{align*}
\zeta_L(s) 
&= f_L(\MAX) + \sum_{t\in M:t\geq s} f_L(t) \\
&= 0 + \sum_{t\in M:t\geq s} f(t) \\
&= f_\zeta(s),
\end{align*}
and, since the M\"obius function of a partial order at an ordered pair $(a,b)$ depends only on the interval $[a,b]$ in the partial order, we have
\begin{align*}
\mu_L(s) 
&= \mu_L(s,\MAX)f_L(\MAX) + \sum_{t\in M: t\geq s} \mu_L(s,t)f_L(t)\\
&= \mu_L(s,\MAX)\cdot 0 + \sum_{t\in M:t\geq s} \mu(s,t)f(t)\\
&= f_\mu(s).
\end{align*}
Thus we can compute the collections of coefficients $\{f_\zeta(s):s\in M\}$ and $\{ f_\mu(s):s\in M \}$ each in $O(|L|v) = O((|M|+1)v) = O(|M|v + |M|) = O(|M|v)$ (since $v\leq |M|$) operations, as claimed.
\end{proof}

\begin{rmk}Let $E(S)$ denote the set of idempotents of $S$. 
It is easy to see that $E(S)$ is a sub-inverse semigroup of $S$ and $(E(S),\leq)$ is a meet-semilattice (where the meet $e \meet f$ of $e,f\in E(S)$ is simply given by $e\meet f = ef=fe$). 
However, in general $(S,\leq)$ itself is not a meet-semilattice. For example, if $S$ is a group with $|S|>1$, then the partial order on $S$ reduces to equality, so $(S,\leq)$ is not a meet-semilattice. More generally, if $e\in E(S)$ is the minimal element and $|G_e|>1$, then $(S,\leq)$ is not a meet-semilattice. However, in many cases of interest $(S,\leq)$ {\em is} a meet-semilattice,
and when this is the case Theorem \ref{ThmBjork} can be used to help bound the terms $\Cplx(\zeta_S)$ and $\Cplx(\mu_S)$ in Theorem \ref{ThmFFTMainResult}, Theorem \ref{ThmMainInv}, and Corollary \ref{CorMainInv}. 
\end{rmk}



\section{FFTs and FIFTs for specific classes of inverse semigroups}

\label{FastInversionSpecificISGS}

In \cite{InvSemiFFT} we developed fast Fourier transforms for the rook monoid and its wreath product by arbitrary finite groups. In this section we begin by giving fast Fourier inversion algorithms for these semigroups relative to the same sets of representations used in \cite{InvSemiFFT}, so our results in this section also give rise to fast convolution algorithms for these semigroups. We then proceed to give fast forward and inverse Fourier transform algorithms for other inverse semigroups of interest not previously considered---namely, for the planar rook monoid, the partial cyclic shift monoid, and the partial rotation monoid.

\subsection{The rook monoid} 
\label{SecRookFIFT}
In \cite{InvSemiFFT} we used the approach in Section \ref{SubSecEvalFT} to show that, for any
$f=\sum_{s\in R_n} f(s)s \in \C R_n$, the change of basis from the semigroup basis to the groupoid basis can be computed in no more than 
$
\frac{2}{3}n^3 |R_n|
$
operations, and the change of basis from the groupoid basis to the Fourier basis of $\C R_n$ relative to $\Y$ can be computed in no more than 
$
\frac{3}{4} n (n-1) |R_n| 
$
operations, where $\Y$ is the set of representations of $\C R_n$ induced by Young's seminormal (or orthogonal) representations of the symmetric group.
Since $n=O(\log |R_n|)$, it followed that the Fourier transform of an arbitrary $\C$-valued function on $R_n$ can be computed in
$
O(|R_n| \log ^3 |R_n|)
$
operations. 

Bj\"orklund et al.\ \cite{Bjork} then showed that the computation of the change of basis from the semigroup basis to the groupoid basis of $\C R_n$ requires no more than $O(|R_n| \log^2 |R_n|)$ operations---this proves the following theorem.

\begin{thm}
Let $f=\sum_{s\in R_n}f(s)s \in \C R_n$. Let $\Y$ be the complete set of inequivalent irreducible representations of $\C R_n$ induced by Young's seminormal or orthogonal representations of the symmetric group. Then the Fourier transform of $f$ relative to $\Y$ can be computed in no more than $O(|R_n| \log^2 |R_n|)$ operations.
\end{thm}

We now show that a similar result holds for Fourier inversion in $R_n$. In particular, we have:

\begin{thm}
\label{RnInversionResult}
Let $f=\sum_{s\in R_n}f(s)s \in \C R_n$. Let $\Y$ be the complete set of inequivalent, irreducible representations of $\C R_n$ induced by Young's seminormal or orthogonal representations of the symmetric group. 
If $f\in \C R_n$ is expressed with respect to $\Y$, then we can compute the coefficients $\{f(s):s\in R_n\}$ such that $f=\sum_{s\in R_n}f(s)s$ in no more than $O(|R_n| \log^2 |R_n|)$ operations.
\end{thm}

\begin{proof}Let $D_0 ,\ldots, D_n$ denote the $\D$-classes of $R_n$, where $D_k$ is the set of elements of $R_n$ of rank $k$. Since the idempotents of $R_n$ are the restrictions of the identity map, $D_k$ contains $\binom{n}{k}$ idempotents and if $e\in D_k$ is idempotent, then $G_e\cong S_k$. Let $\Y$ be the induced set of representations of $\C R_n$ given by taking $e_k$ to be the partial identity on $\{1,2,\ldots,k\}$, taking $p_a$ (for any idempotent $a\in D_k$) to be the unique order-preserving bijection from $\ran(e_k)=\dom(e_k)$ to $\ran(a)=\dom(a)$, and taking $\IRR(G_k)$ to be Young's seminormal or orthogonal representations of $\C S_k$.
By Theorem \ref{ThmMainInv}, then, the number of operations needed to compute the coefficients $\{f(s):s\in R_n\}$ is no more than
\[
\Cplx(\mu_{R_n}) + \sum_{k=0}^n \binom{n}{k}^2 \Tinv(\IRR(S_k)).
\]
Maslen's algorithm \cite{Maslen} for Fourier inversion on $\C S_k$ implies that
\[
\Tinv(\IRR(S_k)) \leq \frac{3}{4} k(k-1) k!,
\]
so we need no more than
\begin{align*}
\Cplx(\mu_{R_n}) + \sum_{k=0}^n \binom{n}{k}^2 \frac{3}{4} k (k-1) k! & \leq \Cplx(\mu_{R_n}) + \frac{3}{4}n (n-1) \sum_{k=0}^n \binom{n}{k}^2  k!\\
& = \Cplx(\mu_{R_n}) + \frac{3}{4}n (n-1) |R_n|\\
& = \Cplx(\mu_{R_n}) + O(|R_n| \log^2 |R_n|)
\end{align*}
operations to compute the coefficients $\{f(s):s\in R_n\}$. 

To handle the $\Cplx(\mu_{R_n})$ term, we note that $(R_n,\leq)$ is a meet-semilattice, where the meet $\sigma\meet\tau\in R_n$ of two elements $\sigma,\tau\in R_n$ is the maximal common restriction between $\sigma$ and $\tau$---namely, $\sigma\meet\tau$ is the element such that 
$$
\dom(\sigma\meet\tau) = \{x\in \{1,2,\ldots,n\} : x\in\dom(\sigma) \cap \dom(\tau) \textup{ and }\sigma(x)=\tau(x)\},
$$
and for $x\in \dom(\sigma\meet\tau)$, we have $\sigma\meet\tau(x)=\sigma(x)=\tau(x)$. Let $L$ denote $(R_n,\leq)$ with a formal maximal element adjoined and let $n\geq 1$. The join-irreducibles of $L$ are the elements of $R_n$ of rank 1, of which there are $n^2$. Theorem \ref{ThmBjork} then applies, and yields $\Cplx(\mu_{R_n}) = O(n^2|R_n|) = O(|R_n|\log^2 |R_n|)$.
\end{proof}

\subsection{Rook wreath products} 
\label{SecWreathFIFT}
We now consider wreath products of $R_n$ with arbitrary finite groups.

\begin{defn}
\label{DefRookWreath}
If $G$ is a finite group, then the {\em rook wreath product} $\Gwr R_n$ is the semigroup of all $n\times n$ matrices with entries in $G \uplus \{0\}$ having at most one entry from $G$ in each row and column under the operation of matrix multiplication (extended from the multiplication of $G$), where $0 g=g 0=0$ for all $g\in G$.
\end{defn}

Write $1$ for the identity of $G$. Clearly we recover $R_n$ as $\Z_1\wr R_n$. In \cite{InvSemiFFT} we showed that, if $G$ is an arbitrary finite group, then the Fourier transform of an arbitrary $\C$-valued function on $\Gwr R_n$ can be computed in $O(|\Gwr R_n| \log^4 |\Gwr R_n|)$ operations. We now show that a similar result holds for Fourier inversion for $\Gwr R_n$.

\begin{thm}
\label{ThmRnWreathInversion}
There exists a complete set $\Y$ of inequivalent irreducible representations of $\C \Gwr R_n$ such that the Fourier transform and the inverse Fourier transform relative to $\Y$ of an arbitrary element $f\in \C \Gwr R_n$ can be computed in $O(|\Gwr R_n| \log^4 |\Gwr R_n|)$ operations.
\end{thm}

\begin{proof}Recall that the {\em symmetric group wreath product} $\Gwr S_k$ is group of $k\times k$ matrices with entries in $G\uplus\{0\}$ with exactly one entry from $G$ in each row and column. For $x\in \Gwr R_n$, let $\rk(x)$ denote the number of rows (or columns) of $x$ with an entry in $G$. The idempotents of $\Gwr R_n$ are the restrictions of the identity matrix, and if $e\in \Gwr R_n$ is idempotent with $\rk(e)=k$, then the maximal subgroup $G_e$ at $e$ is isomorphic to $\Gwr S_k$ \cite{InvSemiFFT}.

The natural partial order $\leq$ on $\Gwr R_n$ can be described in the following manner. For $s,t\in \Gwr R_n$, we have $s\leq t$ if and only if $s$ may be obtained from $t$ by replacing entries of $t$ with $0$. The $\D$-classes of $\Gwr R_n$ are $D_0 ,\ldots, D_n$, where $D_k=\{x\in \Gwr R_n:\rk(x)=k\}$, so $D_k$ contains $\binom{n}{k}$ idempotents \cite{InvSemiFFT}. It is easy to see that $|\Gwr S_k|=k!|G|^k$, 
\[
|\Gwr R_n| = \sum_{k=0}^n \binom{n}{k}^2 k! |G|^k,
\]
and $n=O(\log |\Gwr R_n|)$.

Let $\IRR(\Gwr S_k)$ denote the complete set of inequivalent, irreducible representations of the symmetric group wreath product algebra $\C \Gwr S_k$ constructed in \cite{DanWreath}. Let $\Y$ be the set of representations of $\C \Gwr R_n$ induced by the $\IRR(\Gwr S_k)$, by taking $e_k$ to be the partial identity on $\{1,2,\ldots,k\}$ (thought of as a rook matrix), and taking $p_a$ (for any idempotent $a\in D_k$) to be the unique order-preserving bijection from $\ran(e_k)=\dom(e_k)$ to $\ran(a)=\dom(a)$ (thought of as a rook matrix). $\Y$ is the set of representations we used in \cite{InvSemiFFT} for constructing our $O(|\Gwr R_n| \log^4 |\Gwr R_n|)$-complexity Fourier transform, and we now show that the inverse Fourier transform of an arbitrary element $f\in \C \Gwr R_n$ expressed relative to $\Y$ can also be computed in the stated number of operations.

In \cite{DanWreath} it is shown that, if $G$ has $h$ conjugacy classes, then
$$\Tinv(\IRR(\Gwr S_k)) \leq |\Gwr S_k| \left( |G|\frac{k(k+1)}{2} + 2^h\frac{k^2(k+1)^2}{4}+1  \right).$$
Note that $|G|$ and $h$ are constants with respect to $n$. By Theorem \ref{ThmMainInv}, the number of operations required to compute the inverse Fourier transform of $f$ is no more than  
\begin{align*}
& \Cplx(\mu_{\Gwr R_n}) + \sum_{k=0}^n \binom{n}{k}^2 |\Gwr S_k| \left( |G|\frac{k(k+1)}{2} + 2^h\frac{k^2(k+1)^2}{4}+1  \right) \\
&\leq \Cplx(\mu_{\Gwr R_n}) + \left( |G|\frac{n(n+1)}{2} + 2^h\frac{n^2(n+1)^2}{4}+1  \right) \sum_{k=0}^n \binom{n}{k}^2 k!|G|^k \\
&\leq \Cplx(\mu_{\Gwr R_n}) + O(|\Gwr R_n| \log^4 |\Gwr R_n|).
\end{align*}

To handle the $\Cplx(\mu_{\Gwr R_n})$ term, we note that  $\Gwr R_n$ is a meet-semilattice, where the meet $x\meet y$ of two elements $x,y\in \Gwr R_n$ is given by the maximal common restriction of $x$ and $y$. Specifically, the rows and columns of the elements of $\Gwr R_n$ are indexed by $\{1,2,\ldots,n\}$, and for $i,j\in \{1,2,\ldots,n\}$,
\[
(x\meet y)_{i,j}=
\begin{cases}
x_{i,j} &\textup{if } x_{i,j}=y_{i,j}; \\
0 &\textup{otherwise.}
\end{cases}
\]
Let $n\geq 1$ and let $L$ denote $(\Gwr R_n,\leq)$ with a formal maximal element adjoined. The join-irreducibles of $L$ are the elements of $\Gwr R_n$ with exactly one entry in $G$, of which there are $|G|n^2$. Theorem \ref{ThmBjork} applies, and yields $\Cplx(\mu_{\Gwr R_n}) = O(|\Gwr R_n||G|n^2) = O(|\Gwr R_n| \log^2 |\Gwr R_n|)$.
\end{proof}

\subsection{The planar rook monoid} 
\label{SecPlanarRookFFT}
Any element $\sigma\in R_n$ can be represented by a graph consisting of two rows of $n$ vertices, where vertex $i$ in the first row is connected by a line segment to vertex $j$ in the second row if $\sigma(i)=j$. Call $\sigma\in R_n$ {\em planar} if this representation of $\sigma$ has no crossing edges. The {\em planar rook monoid} $P_n$ is the submonoid of $R_n$ consisting of its planar elements. Equivalently, $P_n$ is the collection of order-preserving injective partial functions from $\{1,2,\ldots,n\}$ to $\{1,2,\ldots,n\}$. Since the composition of two planar elements is planar and the inverse of a planar element is planar, $P_n$ is an inverse semigroup. The representation theory of $P_n$ was worked out in \cite{PlanarRook}. Here we approach the representation theory of $P_n$ through Theorem \ref{ThmRepInducing} to prove the following theorem.

\begin{thm}
\label{ThmPlanarRookFFT}
There exists a complete set $\Y$ of inequivalent irreducible representations of $\C P_n$ such that the Fourier transform and the inverse Fourier transform relative to $\Y$ of an arbitrary element $f\in \C P_n$ can be computed in $O(|P_n|\log^2 |P_n|)$ operations.
\end{thm}

\begin{proof}For $s,t\in P_n$, if $\rk(s)\neq \rk(t)$ then $s$ and $t$ are certainly not $\D$-related. If $\rk(s)=\rk(t)$, then taking $x\in P_n$ to be the unique order-preserving bijection from $\ran(s)$ to $\ran(t)$ shows that $s\D t$. The $\D$-classes of $P_n$ are therefore $D_0 ,\ldots, D_n$, where $D_k$ is the set of elements of $P_n$ of rank $k$. 

The idempotents of $P_n$ are precisely those of $R_n$ (so $D_k$ has $\binom{n}{k}$ idempotents), and for all $k$, if $e\in D_k$ is idempotent, then $G_e\cong \Z_1$. It follows that
\[
|P_n| = \sum_{k=0}^n \binom{n}{k}^2,
\]
so $n=O(\log |P_n|)$. Let $e_k$ denote the partial identity on $\{1,2,\ldots,k\}$, so $G_k\cong \Z_1$ (whose Fourier transform is trivial), and for any idempotent $a\in D_k$, let $p_a$ be the unique order-preserving bijection from $\ran(e_k)$ to $\ran(a)$. Let $\Y$ denote the associated set of induced representations of $\C P_n$. We note that $\Y$ coincides with the representations written down in \cite{PlanarRook}. Theorems \ref{ThmFFTMainResult} and \ref{ThmMainInv} then imply that the Fourier transform and inverse Fourier transform relative to $\Y$ of an arbitrary element $f\in \C P_n$ can be computed in $\Cplx(\zeta_{P_n})$ and $\Cplx(\mu_{P_n})$ operations, respectively.

$(P_n,\leq)$ is a meet-semilattice, where the meet $x\meet y$ of $x,y\in P_n$ is the meet of $x,y$ in $R_n$. (See the proof of Theorem \ref{RnInversionResult}. The main observation here is simply that the restriction of a planar map is planar.) Let $L$ denote the lattice obtained by adjoining a formal maximal element to $P_n$. The join-irreducibles of $L$ are the $n^2$ elements of $P_n$ of rank one, so by Theorem \ref{ThmBjork} we have $\Cplx(\zeta_{P_n}), \Cplx(\mu_{P_n}) = O(|P_n|n^2) = O(|P_n|\log^2 |P_n|)$.
\end{proof}

\subsection{Inverse semigroup generalizations of the cyclic group}
\label{SecCnAndRotn}

We now define and study the Fourier transform on two natural inverse semigroup analogues of the cyclic group $\Z_n$, whose Fourier transform has been enormously useful in applications. We call these analogues the {\em partial cyclic shift monoid} and the {\em partial rotation monoid}. 
$\Z_n$ can be viewed as the group of cyclic shifts of an $n$-set or, equivalently, the group of rotations of $n$ equally spaced points on a circle. In this section we take our set of equivalence class representatives for the integers mod $n$ to be $\{1,2,\ldots,n\}$. Following Lawson \cite[pp.\ 17]{Lawson}, we regard inverse semigroups as collections of partial symmetries, where a partial symmetry of a structure is a structure-preserving bijection between two of its subsets. This motivates the following definitions.

\begin{defn}Let $S,T\subseteq\{1,2,\ldots,n\}$ with $|S|=|T|$. Say $S=\{s_1<s_2<\cdots < s_k\}$ and $T=\{t_1 < t_2 < \cdots < t_k\}$. We say $\sigma$ is a {\em cyclic shift from $S$ to $T$} if $\sigma:S\rightarrow T$ a bijection, where $\sigma(s_1)=t_j$ implies $\sigma(s_{r}) = t_{j+r-1 \pmod{k}}$ for all $r\in \{1,2,\ldots,k\}$.
\end{defn}
We note that the empty bijection is a cyclic shift, and that $\Z_n$ is the group of cyclic shifts from $\{1,2,\ldots,n\}$ to $\{1,2,\ldots,n\}$.

\begin{defn}The {\em partial cyclic shift monoid} $C_n$ is the subset of $R_n$ consisting of all cyclic shifts from $S$ to $T$, as $S$ and $T$ range across the subsets of $\{1,2,\ldots,n\}$.
\end{defn}

Note that we only have a cyclic shift from $S$ to $T$ if $|S|=|T|$. Of course, the identity of $C_n$ is the identity of $R_n$.

\begin{prop}$C_n$ is an inverse semigroup.
\end{prop}

\begin{proof}Since $C_n\subseteq R_n$, it suffices to show that $C_n$ is closed under composition and inverses. It is clear that the inverse of a cyclic shift is a cyclic shift. To show $C_n$ is closed under composition, we begin by noting that the restriction of cyclic shift is a cyclic shift---that is, if $e\in R_n$ is idempotent and $\sigma\in C_n$, then $\sigma e \in C_n$. We also note that if $\tau$ is a cyclic shift from $A$ to $B$ and $\sigma$ is a cyclic shift from $B$ to $C$ (where $A,B,C \subseteq \{1,2,\ldots,n\}$ with $|A|=|B|=|C|$), then $\sigma \tau$ is a cyclic shift from $A$ to $C$. 

If $\sigma\in R_n$ and $S\subseteq \{1,2,\ldots,n\}$, let $\sigma|_{S}$ denote $\sigma e$, where $e$ is the partial identity on $S$. That is, $\sigma|_{S}$ is the map given by restricting the domain of $\sigma$ to $\dom(\sigma)\cap S$.
Suppose then that $\sigma$ and $\tau$ are cyclic shifts. Let $\tau'=\tau|_{\dom(\sigma \tau)}$ and $\sigma ' = \sigma |_{\ran(\tau')}$. Then $\sigma'$ and $\tau'$ are cyclic shifts, $\sigma \tau = \sigma' \tau '$, and $\dom(\sigma') = \ran(\tau')$. It follows that $\sigma \tau$ is a cyclic shift (from $\dom(\sigma\tau)=\dom(\tau')$ to $\ran(\sigma\tau)=\ran(\sigma')$).
\end{proof}

Thus $C_n$ is an inverse semigroup analogue of $\Z_n$ in the setting of partial symmetries. We analyze the Fourier transform on $C_n$ in Section \ref{SubsecPartialCyclicShift}.

Next we consider partial rotations, which lead to a different inverse semigroup analogue of $\Z_n$. Place equally-spaced points $\{1,2,\ldots,n\}$ along the perimeter of a circle, and view $\Z_n$ as the group of bijections from $\{1,2,\ldots,n\}$ to $\{1,2,\ldots,n\}$ by rotation of the circle. A {\em partial rotation} is a bijection between subsets $S,T\subseteq \{1,2,\ldots,n\}$ obtained by the restriction of a rotation. Specifically:

\begin{defn}Let $r\in S_n\subseteq R_n$ be the $n$-cycle given by $r(i)=i+1 \pmod n$. Let $\sigma\in R_n$. We say $\sigma$ is a {\em partial rotation} if $\sigma=r^k f$ for some $k\in \Z$ and some idempotent $f\in R_n$. The {\em partial rotation monoid} $\Rot_n$ is the set of partial rotations in $R_n$. 
\end{defn}

Let $e$ denote the identity of $R_n$. Clearly $e$ is the identity of $\Rot_n$, and we find $\Z_n$ in $\Rot_n$ as the set of elements of the form $r^ke$.

\begin{lem}
\label{LemRotn_rk_on_either_side}
Let $\sigma\in R_n$. Then $\sigma\in \Rot_n$ if and only if $\sigma=gr^k$ for some $k\in \Z$ and some idempotent $g\in R_n$. Furthermore, if $\sigma \in \Rot_n$, then $\sigma=r^kf=er^k$ for some $k\in \Z$ and idempotents $e,f\in R_n$. 
\end{lem}
\begin{proof}
If $\sigma=r^k f$ for some idempotent $f\in R_n$, then $\sigma=g r^k$ for the idempotent $g=r^{k} f r^{-k}$, and if $\sigma=gr^k$ then $\sigma=r^k f$ for the idempotent $f=r^{-k} g r^k$. 
\end{proof}

\begin{prop}$\Rot_n$ is an inverse semigroup.
\end{prop}

\begin{proof}Since $\Rot_n\subseteq R_n$, it suffices to show that $\Rot_n$ is closed under inverses and composition. From Lemma \ref{LemRotn_rk_on_either_side} it follows that the inverse $\sigma^{-1}=f r^{-k}$ of a partial rotation $\sigma=r^k f$ is a partial rotation. To show $\Rot_n$ is closed under composition, if $\sigma=gr^k$ and $\tau = r^j f$ for idempotents $f,g\in R_n$, then $\sigma \tau = gr^kr^j f = (gr^{k+j})f = (r^{k+j}h) f = r^{k+j}(hf)$ for the idempotent $h = r^{-k-j}gr^{k+j}$. Since $h$ and $f$ are idempotent, so is $hf$, so $\sigma \tau$ is a partial rotation.
\end{proof}

Therefore, $\Rot_n$ is also an inverse semigroup analogue of $\Z_n$ in the setting of partial symmetries. We analyze the Fourier transform on $\Rot_n$ in Section \ref{SubsecPartialRotation}.

\begin{rmk}We note that $\Rot_n\subseteq C_n$. We also note that, even though $C_n$ and $\Rot_n$ are inverse semigroup generalizations of $\Z_n$ and the maximal subgroups of $C_n$ and $\Rot_n$ are all abelian, $C_n$ and $\Rot_n$ are themselves non-abelian for $n\geq 2$.
\end{rmk}

\subsubsection{The Fourier transform on the partial cyclic shift monoid}

\label{SubsecPartialCyclicShift} 

We now analyze the complexity of the Fourier transform on $C_n$.

\begin{thm}
\label{ThmFFTCn}
There exists a complete set of inequivalent, irreducible representations $\Y$ of $\C C_n$ such that the Fourier transform and the inverse Fourier transform relative to $\Y$ of an arbitrary element $f\in \C C_n$ can be computed in $O(|C_n| \log^2 |C_n|)$ operations. 
\end{thm}

\begin{proof}It is clear that the idempotents of $C_n$ are precisely those of $R_n$. For $s,t\in C_n$, if $\rk(s)\neq \rk(t)$ then $s$ and $t$ are not $\D$-related, and if $\rk(s)=\rk(t)$ then taking $x\in C_n$ to be the unique order-preserving bijection from $\ran(s)$ to $\ran(t)$ shows that $s\D t$. The $\D$-classes of $C_n$ are therefore $D_0 ,\ldots, D_n$, where $D_k$ consists of the rank-$k$ elements of $C_n$. It follows that $D_k$ contains $\binom{n}{k}$ idempotents. If $e\in D_k$ is idempotent, then $\rk(e)=k$ and $G_e$ is the set of all cyclic shifts from $\dom(e)$ to $\dom(e)$, which is isomorphic to $\Z_k$. 

Let $\IRR(\Z_k)$ denote a complete set of inequivalent, irreducible representations of $\C \Z_k$. (Equivalently, $\IRR(\Z_k)$ is the set of characters of $\Z_k$.) Let $\Y$ be the set of representations of $\C C_n$ induced by the $\IRR(\Z_k)$, by taking $e_k$ to be the partial identity on $\{1,2,\ldots,k\}$ and taking $p_a$ (for any idempotent $a\in D_k$) to be the unique order-preserving bijection from $\ran(e_k)$ to $\ran(a)$. We have
\[
|C_n| = 1+ \sum_{k=1}^n \binom{n}{k}^2 k,
\]
so $n=O(\log |C_n|)$. Theorems \ref{ThmFFTMainResult} and \ref{ThmMainInv} imply that the Fourier transform and inverse Fourier transform relative to $\Y$ of an arbitrary element $f\in \C C_n$ can be computed in
\[
\Cplx(\zeta_{C_n}) + \sum_{k=0}^n \binom{n}{k}^2 \T(\IRR(\Z_k))
\]
and
\begin{equation}
\label{eqCnPf2}
\Cplx(\mu_{C_n}) + \sum_{k=0}^n \binom{n}{k}^2 \Tinv(\IRR(\Z_k))
\end{equation}
operations, respectively. It is well known that there is a constant $c$ for which $\T(\IRR(\Z_k))\leq c k\log k$ and $\Tinv(\IRR(\Z_k)) \leq c k\log k$ for all $k$ \cite{Bluestein}. Thus
\begin{align*}
\Cplx(\zeta_{C_n}) + \sum_{k=0}^n \binom{n}{k}^2 \T(\IRR(\Z_k)) & \leq
\Cplx(\zeta_{C_n}) + \sum_{k=1}^n \binom{n}{k}^2 ck \log k \\ &\leq
\Cplx(\zeta_{C_n}) + c\log n \sum_{k=1}^n \binom{n}{k}^2 k \\ &=
\Cplx(\zeta_{C_n}) + O(|C_n| \log n)\\ &=
\Cplx(\zeta_{C_n}) + O(|C_n| \log \log |C_n|).
\end{align*}
Similarly, \eqref{eqCnPf2} is $\Cplx(\mu_{C_n}) + O(|C_n| \log\log |C_n|)$. 

To handle $\Cplx(\zeta_{C_n})$ and $\Cplx(\mu_{C_n})$, $(C_n,\leq)$ is a meet-semilattice, where the meet $s\meet t\in C_n$ of two elements $s,t\in C_n$ is simply the meet of $s,t$ in $R_n$. (See the proof of Theorem \ref{RnInversionResult}. The main observation here is that the restriction of a cyclic shift is a cyclic shift.) Let $L$ denote the lattice obtained by adjoining a formal maximal element to $C_n$. The join-irreducibles of $L$ are the $n^2$ elements of $C_n$ of rank one, so by Theorem \ref{ThmBjork} we have $\Cplx(\zeta_{C_n}), \Cplx(\mu_{C_n}) = O(|C_n|n^2) = O(|C_n|\log^2 |C_n|)$.
\end{proof}

\subsubsection{The Fourier transform on the partial rotation monoid} 

\label{SubsecPartialRotation}

In this section we analyze the complexity of the Fourier transform on $\Rot_n$. Before proceeding, we note that our analysis of $R_n$, $\Gwr R_n$, $P_n$ and $C_n$ thus far have been quite similar. The main reason for this is that all of the semigroups analyzed so far contain the unique order-preserving bijection from $A$ to $B$, for all $A,B \subseteq \{1,2,\ldots,n\}$ with $|A|=|B|$. This causes each of these semigroups to have $\D$-classes $D_0 ,\ldots, D_n$, where $D_k$ is the set of elements of the semigroup of rank $k$. If $S\subseteq R_n$ is an inverse semigroup, then it is clear that $x,y\in S$ can only be $\D$-related if $\rk(x)=\rk(y)$, so in this sense $R_n$, $\Gwr R_n$, $P_n$, and $C_n$ have the fewest $\D$-classes possible. Our analysis of $\Rot_n$ is different, because the unique order-preserving bijection from $A$ to $B$ is not necessarily a partial rotation. This causes $\Rot_n$ to have more than $n+1$ $\D$-classes in general.

Let $r\in \Rot_n$ be the $n$-cycle given by $r(i)=i+1 \pmod n$, and let us identify $\Z_n$ with the subgroup of $\Rot_n$ generated by $r$. (That is, we identify $\Z_n$ with the set of elements of $\Rot_n$ of full rank.) Let $\Z_n\subseteq \Rot_n$ act on the subsets of $\{1,2,\ldots,n\}$ by rotation. Denote this action by $\cdot$, so for $r^k\in \Z_n$ we have $r^k \cdot \{t_1 , t_2, \ldots, t_i\} = \{t_1+k, t_2+k, ,\ldots, t_i + k\}$, where all sums are taken mod $n$. As in Section \ref{SecMatrixAlgsOverGpAlgs}, we identify the subsets of $\{1,2,\ldots,n\}$ with the partial identities on these subsets in $\Rot_n$. It is important to note that $\cdot$ does {\em not} coincide with the multiplication in $\Rot_n$. That is, if $e\in \Rot_n$ is idempotent, then $r^k \cdot e \neq r^k e$ in general. Rather, it is straightforward to check that $r^k\cdot e=r^k e r^{-k}$. We will write the operation on $\Rot_n$ as concatenation, while reserving $\cdot$ to refer only to the action of $\Z_n$. Here is a characterization of $\D$ on $\Rot_n$.

\begin{lem}
\label{LemRot}
Let $a,b\in \Rot_n$. Then $a\D b$ if and only if there exists $k\in \Z$ such that $r^k \cdot \ran(a) = \ran(b)$. 
\end{lem}

\begin{proof}
Let $a\D b$. Let $x=r^ke=fr^k \in \Rot_n$ (for some $k\in \Z$ and idempotents $e,f \in \Rot_n$) such that $\dom(x)=\ran(a)$ and $\ran(x)=\ran(b)$. Then $\dom(x)=e$ and $\ran(x)=f$, and $r^ker^{-k}=f=r^k \cdot e$, so $r^k \cdot \ran(a) = \ran(b)$.

On the other hand, let $r^k \cdot \ran(a) = \ran(b)$. Then, for $x=r^k \ran(a)$, we have $\dom(x)=\ran(a)$ and $\ran(x) = r^k \ran(a) \ran(a) r^{-k} = r^k \ran(a) r^{-k} = r^k \cdot \ran(a) = \ran(b),$ so $a\D b$.
\end{proof}

For any idempotent $e\in \Rot_n$, let $j(e)$ be the smallest positive integer $j$ such that $r^j \cdot e = e$. Equivalently, $j(e)$ is the size of the orbit of $e$ under $\cdot$. By the division algorithm (or the orbit-stabilizer theorem), $j(e)$ divides $n$. 

\begin{lem}
\label{LemRotnIdemsPerDClass}
If $e\in \Rot_n$ is idempotent, then the $\D$-class of $e$ contains $j(e)$ idempotents.
\end{lem}
\begin{proof}By Lemma \ref{LemRot}, if $f\in \Rot_n$ is idempotent, then $e\D f$ if and only if there exists $k\in \Z$ such that $r^k\cdot e=f$, so the idempotents to which $e$ is $\D$-related are the distinct idempotents $r\cdot e , r^2\cdot e, \ldots, r^{j(e)}\cdot e=e$. 
\end{proof}

\begin{lem}
\label{LemRotnMaxSubgps}
Let $e\in \Rot_n$ be idempotent with $\rk(e)\geq 1$. Then the maximal subgroup of $\Rot_n$ at $e$ is isomorphic to $\Z_{n/j(e)}$. 
\end{lem}

\begin{proof}Let $e\in \Rot_n$ be idempotent and let $j=j(e)$. First we show that the maximal subgroup $G_e$ at $e$ is given by 
\[
G_e = \{r^{jk}e : k\in \Z \}.
\]

From Section \ref{SecMatrixAlgsOverGpAlgs} we have $G_e = \{\sigma\in \Rot_n:\dom(\sigma)=\ran(\sigma) = e\}$. 
Let $k\in \Z$ and let $\sigma=r^{jk}e$. Then $\dom(\sigma)=e$, and $\ran(\sigma)=r^{jk}eer^{-jk} = r^{jk}er^{-jk} = r^{jk}\cdot e = (r^j)^k\cdot e = e$. Therefore $r^{jk} e \in G_e$ for all $k\in \Z$. On the other hand, let $\sigma\in G_e$, so $\sigma=r^qf$ for some $q\in \Z$ and some idempotent $f\in \Rot_n$, with $\dom(\sigma)=e$ and $\ran(\sigma)=e$. Since 
$\dom(\sigma)=f$, we have $f=e$, so $\sigma=r^q e$. Then $e=\ran(\sigma)=r^qe(r^q e)^{-1} = r^q e e r^{-q} = r^q e r^{-q} = r^q \cdot e$. That is, $r^q \cdot e = e$. It is straightforward to show that the minimality of $j$ implies that $j$ divides $q$, so we have $\sigma = r^{jk}e$ for some $k\in \Z$. Thus $G_e = \{r^{jk}e : k\in \Z \}$, as claimed.

Now let $\rk(e)\geq 1$. It is clear that $G_e = \{r^je, r^{2j}e, \ldots, r^{\frac{n}{j}j}e = e\}$, and we claim that the elements in this list are distinct. To see why, suppose not. Then $r^{ji}e=r^{jk}e$ for some $1\leq i<k\leq n/j$. Let $x\in \dom(e)$, so applying $r^{ji}e$ and $r^{jk}e$ to $x$ we have $(r^{ji}e)(x)=(r^{jk}e)(x)$, so $r^{ji}(x)=r^{jk}(x)$, so $x=r^{jk-ji}(x)$, but that is absurd because $r^{jk-ji}$ is a nontrivial rotation. It is now clear that $G_e$ is isomorphic to $\Z_{n/j}$.
\end{proof}

The final ingredient we need is a description of the poset structure of $\Rot_n$. For $k\in \N$ let $B_k$ denote the boolean lattice of subsets of $\{1,2,\ldots,k\}.$ First, it is clear that the order ideal $\{\tau\in \Rot_n:\tau\leq \sigma\}$ is isomorphic to the boolean lattice $B_{\rk(\sigma)}$ for any $\sigma\in \Rot_n$. What is nice is that if $\sigma\in \Rot_n$ and there exists $i\in \dom(\sigma)$, then $\sigma(k)$ is determined for all $k\in \dom(\sigma)$. This means that the order filter $\{\tau\in \Rot_n:\tau\geq \sigma\}$ is isomorphic to the boolean lattice $B_{n-\rk(\sigma)}$ for all $\sigma \in \Rot_n$ with $\rk(\sigma)\geq 1$. It follows that $(\Rot_n,\leq)$ is isomorphic to $n$ disjoint copies of $B_n$---one for each element of $\Rot_n$ of rank $n$---identified at their minimal elements.

We now analyze the complexity of the Fourier transform on $\Rot_n$.

\begin{thm}
\label{ThmRotnCplx}
There exists a complete set of inequivalent, irreducible representations of $\C \Rot_n$ such that the Fourier transform and the inverse Fourier transform relative to $\Y$ of an arbitrary element $f\in \C \Rot_n$ can be computed in $O(|\Rot_n|\log |\Rot_n|)$ operations.
\end{thm}

\begin{proof}
First, we note that the poset description of $\Rot_n$ above implies that $|\Rot_n| = 2^n n - n + 1$, so $n = O(\log |\Rot_n|)$.

Next, since the elements of any $\D$-class of $\Rot_n$ are all of the same rank, for $k=0 ,\ldots, n$, let $d(k)$ denote the number of $\D$-classes of $\Rot_n$ consisting of rank-$k$ elements, and label the $\D$-classes consisting of rank-$k$ elements $D_{k,1}, D_{k,2}, \ldots, D_{k,d(k)}$. Choose an idempotent $e_{k,l}$ for each $\D$-class $D_{k,l}$, and let $j(k,l)=j(e_{k,l})$. 
Then, by Lemma \ref{LemRotnIdemsPerDClass}, $D_{(k,l)}$ has $j(k,l)$ idempotents and, by Lemma \ref{LemRotnMaxSubgps}, for $k>0$ the maximal subgroup at $e_{k,l}$ is isomorphic to $\Z_{n/j(k,l)}$. 
Let $\IRR(\Z_{n/j(k,l)})$ be a complete set of inequivalent, irreducible representations of $\C Z_{n/j(k,l)}$. 

For any idempotent $a\in D_{(k,l)}$, let $p_a = r^m e_{k,l} $, where $m$ is the smallest nonnegative integer such that $r^m \cdot e_{k,l} = a$. 
Let $\Y$ be the set of representations of $\C \Rot_n$ induced by the $\IRR(\Z_{n/j(k,l)})$.

Theorems \ref{ThmFFTMainResult} and \ref{ThmMainInv} imply that the Fourier transform and inverse Fourier transform relative to $\Y$ of an arbitrary element $f\in \C \Rot_n$ can be computed in
\begin{equation}
\label{RotnEqnFTCplx}
\Cplx(\zeta_{\Rot_n}) + \sum_{k=1}^n\sum_{l=1}^k {j(k,l)}^2 \T(\IRR(\Z_{n/j(k,l)}))
\end{equation}
and
\begin{equation}
\label{RotnEqnIFTCplx}
\Cplx(\mu_{\Rot_n}) + \sum_{k=1}^n\sum_{l=1}^k {j(k,l)}^2 \Tinv(\IRR(\Z_{n/j(k,l)}))
\end{equation}
operations, respectively. Let $c$ be a constant such that $\T(\IRR(\Z_k))\leq c k \log k$ and $\Tinv(\IRR(\Z_k))\leq c k \log k$ for all $k$. Then we can bound the sum in \eqref{RotnEqnFTCplx} by
\begin{align*}
\sum_{k=1}^n\sum_{l=1}^k {j(k,l)}^2 \T(\IRR(\Z_{n/j(k,l)}))  
&\leq \sum_{k=1}^n\sum_{l=1}^k {j(k,l)}^2 c \frac{n}{j(k,l)} \log\left( \frac{n}{j(k,l)} \right) \\
& = cn \sum_{k=1}^n\sum_{l=1}^k {j(k,l)} \log\left( \frac{n}{j(k,l)} \right) \\
& \leq  cn\log(n) \sum_{k=1}^n\sum_{l=1}^k {j(k,l)} \\
& = cn\log(n) \sum_{k=1}^n \binom{n}{k} \\
& = cn\log(n) (2^n - 1) \\
& \leq c\log(n) (2^n n - n +1) \\
& = O(|\Rot_n|\log \log |\Rot_n|).
\end{align*}
Similarly, in \eqref{RotnEqnIFTCplx} we have
\[
\sum_{k=1}^n\sum_{l=1}^k {j(k,l)}^2 \Tinv(\IRR(\Z_{n/j(k,l)})) = O(|\Rot_n|\log \log |\Rot_n|).
\]

Although it is possible to use Theorem \ref{ThmBjork} to show $\Cplx(\zeta_{\Rot_n}) = O(n^2 |\Rot_n|) = O(|\Rot_n| \log^2 |\Rot_n|)$ (and similarly for $\Cplx(\mu_{\Rot_n})$), 
the following more direct approach yields a better result: $(\Rot_n,\leq)$ is isomorphic to $n$ disjoint copies of the boolean lattice $B_n$ identified at their minimal elements. Fast zeta and M\"obius transforms on $B_n$ are simple to describe and implement---see, e.g., Section 2.2 of \cite{FourierMeetsMobius}. In particular, the zeta or M\"obius transform of an arbitrary $\C$-valued function on $B_n$ can be computed in no more than $n 2^n$ operations. 

Suppose $f:\Rot_n\rightarrow \C$, and for $i\in\{0 ,\ldots, n-1\}$, let $\iota_i(\Rot_n)$ denote $\{\sigma\in \Rot_n: \sigma\leq r^i\}$. We may compute the zeta transform $f_\zeta$ of $f$ in the following manner: First compute the zeta transform of $f$ restricted to each of the $\iota_i(\Rot_n)$, and call the results $f_{\zeta,i}$. Then, for $\sigma\in \Rot_n$, if $\rk(\sigma)\geq 1$, we have $f_\zeta(\sigma) = f_{\zeta,i}(\sigma)$, where $i\in \{0 ,\ldots, n-1\}$ is the unique value $i$ for which $\sigma\leq r^i$. For the element $\sigma\in \Rot_n$ of rank $0$, we have $f_\zeta(\sigma) = (1-n)f(\sigma) + \sum_{i=0}^{n-1} f_{\zeta,i}(\sigma)$. Using fast zeta transforms for the $\iota_i(\Rot_n)$, we have
\begin{align*}
\Cplx(\zeta_{\Rot_n}) &\leq n (n2^n) +  n+1\\
&= n(n2^n - n +1) +n^2+1\\
&=O(n|\Rot_n|)\\
&=O(|\Rot_n|\log|\Rot_n|).
\end{align*}

In a similar fashion, we may compute the M\"obius transform $f_u$ of $f$ in the following manner: First compute the M\"obius transform of $f$ restricted to each of the $\iota_i(\Rot_n)$ and call the results $f_{\mu,i}$. For $\sigma \in \Rot_n$, if $\rk(\sigma)\geq 1$, we have $f_\mu(\sigma)=f_{\mu,i}(\sigma)$, where $i\in \{0 ,\ldots, n-1\}$ is the unique value for which $\sigma \leq r^i$. For the element $\sigma\in \Rot_n$ of rank $0$, we have $f_\mu(\sigma)=(1-n)f(\sigma) + \sum_{i=0}^{n-1}f_{\mu,i}(\sigma)$. Using fast M\"obius transforms for the $\iota_i(\Rot_n)$, we have
\[
\Cplx(\mu_{\Rot_n}) \leq n (n2^n) +  n+1
=O(|\Rot_n|\log|\Rot_n|).
\]
Therefore \eqref{RotnEqnFTCplx} and \eqref{RotnEqnIFTCplx} are both $O(|\Rot_n|\log|\Rot_n|)$.
\end{proof}

\begin{rmk}The changes from the groupoid basis to the Fourier basis for $S = C_n$ and $S = \Rot_n$ in the proofs of Theorems \ref{ThmFFTCn} and \ref{ThmRotnCplx} are accomplished in $O(|S|\log \log |S|)$ operations. If $S$ is a group and all multiplications by constants involved in the computation of the Fourier transform are restricted to multiplications by constants of size no larger than $2$, then it is known that the Fourier transform on $\C S$ requires {\em at least} $\frac{1}{4}|S|\log |S|$ operations \cite{ClausenInversion}. Although our Fourier transforms for $S=C_n$ and $S=\Rot_n$ use $O(|S|\log |S|)$ operations (due to the complexities of the changes of basis from the natural basis to the groupoid basis), $C_n$ and $\Rot_n$ are the first interesting examples of families of inverse semigroups with nontrivial maximal subgroups whose changes of basis from the groupoid basis to the Fourier basis can be achieved in sub-$O(|S|\log|S|)$ complexity.
\end{rmk}

\begin{rmk}
\label{LinearCplx}
Simple examples exist which show that the general $\frac{1}{4}|S|\log |S|$ lower bound on the complexity of the Fourier transform for groups does not extend to inverse semigroups. For example, if $S$ is the chain on $n$ elements under the meet operation, then $S$ is an idempotent inverse semigroup of order $n$, so each $\D$-class of $S$ has size one and the maximal subgroup at each element of $S$ is trivial. Therefore, the Fourier transform of an element $f\in \C S$ is just the zeta transform of $f$, and it is easy to see that the zeta transform of $f\in \C S$ can be computed in linear time. Indeed, let $S=\{s_1 < s_2 < \cdots < s_n\}$ and $f:S\rightarrow \C$. Then set $f_\zeta(s_n) = s_n$ and, for $i=n-1 ,\ldots, 1$, compute $f_\zeta(s_i) = f(s_i) + f_\zeta(s_{i+1})$. Thus we can compute the Fourier transform of $f$ in $n$ operations. The M\"obius transform $f_\mu$ of $f:S\rightarrow \C$ is even simpler. We have $f_\mu(s_n)=f(s_n)$ and, for $i=1,2,\ldots,n-1$, $f_\mu(s_i) = f(s_i)-f(s_{i+1})$, so the inverse Fourier transform of $f$ can also be computed in $n$ operations.
\end{rmk}


\bibliographystyle{plain}
\bibliography{InversionInvSemibib}

\end{document}